\documentclass[a4paper,12pt]{amsart}
\usepackage{paralist} 
\usepackage[T1]{fontenc}
\usepackage{currvita}
\usepackage{geometry}
\usepackage{amsmath}
\allowdisplaybreaks
\usepackage{amssymb}
\usepackage{amsfonts}
\usepackage{color}
\usepackage{enumitem}

\usepackage{booktabs}

\usepackage{etoolbox}
\patchcmd{\thebibliography}{\section*{\refname}}{}{}{}

\usepackage{verbatim}

\usepackage{graphicx}
\usepackage{subfig}
\usepackage{bm}
\usepackage{tikz}
\bgroup

\numberwithin{equation}{section}

\usepackage[bookmarks,bookmarksopen=false,
pdfauthor={Autor},
pdftitle={Titel},
colorlinks=true,
linkcolor=blue,
citecolor=blue,
urlcolor=blue,
hidelinks]{hyperref}

\newtheorem{lemma}{Lemma}[section]
\newtheorem{theorem}[lemma]{Theorem}
\newtheorem{corollary}[lemma]{Corollary}

\newtheorem{rem}[lemma]{Remark}

\DeclareMathOperator*{\supp}{supp}
\DeclareMathOperator{\diag}{diag}
\DeclareMathOperator{\diam}{diam}

\DeclareMathOperator{\sign}{sign}
\DeclareMathOperator{\tr}{tr}

\newcommand{\pbox}{\hfill$\Box$\\}

\newcommand{\R}{\mathbb{R}}

\newcommand{\N}{\mathbb{N}}

\newcommand{\Z}{\mathbb{Z}}

\newcommand{\I}{\mathcal{I}}

\renewcommand{\H}{\mathcal{H}}
\newcommand{\F}{\mathcal{F}}

\newcommand{\ee}{E}
\newcommand{\xx}{X}
\newcommand{\kk}{K}

\definecolor{darkviolet}{rgb}{0.58,0,0.83}

\usetikzlibrary{patterns}

\author[F. Marceca]{Felipe Marceca}
\address[F. Marceca]{Department of Mathematics \\  King’s College London \\ United Kingdom}
\email{felipe.marceca@kcl.ac.uk}

\author[J. L. Romero]{Jos\'e Luis Romero}
\address[J. L. Romero]{Faculty of Mathematics \\
	University of Vienna \\
	Oskar-Morgenstern-Platz 1 \\
	A-1090 Vienna, Austria \\and
	Acoustics Research Institute\\ Austrian Academy of Sciences\\Dr. Ignaz Seipel-Platz 2,	AT-1010 Vienna, Austria}
\email{jose.luis.romero@univie.ac.at}

\author[M. Speckbacher]{Michael Speckbacher}
\address[M. Speckbacher]{Faculty of Mathematics \\
	University of Vienna \\
	Oskar-Morgenstern-Platz 1 \\
	A-1090 Vienna, Austria }
\email{michael.speckbacher@univie.ac.at}

\title[Eigenvalue estimates for Fourier concentration operators]{Eigenvalue estimates for Fourier concentration operators on two domains}

\keywords{Discrete Fourier transform, concentration operator, Hankel operator, eigenvalue, spectrum}

\subjclass[2020]{47B35, 47A75, 42B35, 42C40}

\begin{document}
\begin{abstract}
	We study concentration operators associated with either the discrete or the continuous Fourier transform,
	that is, operators that incorporate a spatial cut-off and a subsequent frequency cut-off to the Fourier inversion formula. Their spectral profiles describe the number of prominent degrees of freedom in problems where functions are assumed to be supported on a certain domain and their Fourier transforms are known or measured on a second domain.
	
	We derive eigenvalue estimates that quantify the extent to which Fourier concentration operators deviate from orthogonal projectors, by bounding the number of eigenvalues that are away from 0 and 1 in terms of the geometry of the spatial and frequency domains, and a factor that grows at most poly-logarithmically on the inverse of the spectral margin. The estimates are non-asymptotic in the sense that they are applicable to concrete domains and spectral thresholds, and almost match asymptotic benchmarks.
	
	Our work covers for the first time non-convex and non-symmetric spatial and frequency concentration domains, as demanded by numerous applications that exploit the expected approximate low dimensionality of the modeled phenomena. The proofs build on Israel's work on one dimensional intervals [arXiv: 1502.04404v1]. The new ingredients are the use of redundant wave-packet expansions and a dyadic decomposition argument to obtain Schatten norm estimates for Hankel operators. 
\end{abstract}

\thanks{The authors gratefully acknowledge support from the Austrian Science Fund (FWF): 10.55776/Y1199. F. M. was also supported by the EPSRC: NIA EP/V002449/1.}

\maketitle
	
\section{Introduction and results}

Fourier concentration operators act by incorporating a spatial cut-off and a subsequent frequency cut-off to the Fourier inversion formula. The chief example concerns the Fourier transform on the Euclidean space $\mathcal{F}: L^2(\R^d) \to L^2(\R^d)$, the cut-offs are then given by the indicator functions of two compact domains $E,F \subseteq \mathbb{R}^d$, and the concentration operator is
\begin{align}\label{eq_S}
	S f=\chi_F \F^{-1}\chi_\ee \F \chi_F f, \qquad f \in L^2(\R^d).
\end{align}
These operators, and their analogues defined with respect to the discrete Fourier transform $L^2([-1/2,1/2]^d) \to \ell^2(\mathbb{Z}^d)$ play a crucial role in many analysis problems and fields of application where the shapes of $E,F$ are dictated by various physical constraints or measurement characteristics \cite{MR140732,MR140733,MR147686,MR2883827}.

The basic intuition, stemming from the Fourier uncertainty principle \cite{MR707957, MR2883827}, is that the concentration operator \eqref{eq_S} is approximately a projection with rank $\tr(S) = |E| \cdot |F|$. The error of such heuristic is encoded by the so-called \emph{plunge region}
\begin{align}\label{eq_Ms}
M_\varepsilon(S)= \{ \lambda \in \sigma(S): \varepsilon < \lambda < 1- \varepsilon\}, \qquad \varepsilon \in (0,1/2),
\end{align}
consisting of intermediate eigenvalues. Asymptotics for the cardinality of $M_\varepsilon(S)$ go back to Landau and Widom \cite{MR593228,MR487600} for the case of one dimensional intervals $E=[-a,a]$, $F=[-b,b]$ and read
\begin{align}\label{eqc}
\# M_\varepsilon(S) = c \cdot \log(ab) \cdot \log(\tfrac{1-\varepsilon}{\varepsilon})
+ o(\log(ab)), \qquad \mbox{as } ab \to \infty,
\end{align}
for an explicit constant $c$ that depends on the normalization of the Fourier transform. The modern spectral theory of Wiener-Hopf operators gives similar asymptotics for concentration operators associated to rather general multi-dimensional domains subject to increasing isotropic dilations.

While \eqref{eqc} precisely describes the cardinality of the set $M_\varepsilon(S)$ in the limit $ab \to \infty$, the asymptotic is often insufficient for many purposes because of the quality of the error terms. Indeed, the error term in \eqref{eqc} depends in an unspecified way on the spectral threshold $\varepsilon$, which precludes applications where $\varepsilon$ is let to vary with the domains $E,F$. Such limitations have motivated a great amount of work aimed at deriving \emph{upper bounds} for $\# M_\varepsilon(S)$ that
are \emph{threshold robust}, that is, bounds that are effective for concrete concentration domains and explicit in their dependence on the spectral threshold \cite{is15,KaRoDa,MR3887338,MR3926960,BoJaKa,Os}, significantly improving on more classical results in this spirit \cite{MR169015}.

With the exception of \cite{is15}, the mentioned articles on threshold-robust spectral bounds for Fourier 
concentration operators concern only the one dimensional case, because they exploit a connection with a Sturm–Liouville equation which is specific to that setting. Such methods can be applied to some extent to higher dimensional domains enjoying special symmetries \cite{MR181766, MR673519}. On the other hand, while \cite{is15} studies Fourier concentration operators associated with one dimensional intervals, the technique introduced by Israel is very general, as it relies on an explicit almost diagonalization of the concentration operator. In fact, as we were finishing this work, the preprint version of \cite{isma23} provided an extension of \cite{is15} to higher dimensions (see Sections \ref{sec_e} and \ref{sec_r}).

In this article we derive upper bounds for the number of intermediate eigenvalues \eqref{eq_Ms} associated with either the continuous or discrete Fourier transforms. In contrast to other results in the literature, we obtain estimates that apply to two suitably regular multi-dimensional spatial and frequency domains, which do not need to exhibit special symmetries. In this way, our results cover for the first time many setups of practical relevance, see Section \ref{sec_sig}.

Our proofs build on Israel's technique \cite{is15} and incorporate novel arguments to treat non-convex domains and their discrete counterparts. Instead of the orthonormal wave-packet basis from \cite{is15}, we use more versatile redundant expansions (frames). Second, we introduce a dyadic decomposition method implemented by means of Schatten norm estimates for Hankel operators, see Section \ref{sec_r}.
	
	\subsection{The Euclidean space}\label{sec_e}
	
	Given two compact sets $E,F \subseteq \mathbb{R}^d$, the \emph{Fourier concentration operator} $S: L^2(\mathbb{R}^d) \to L^2(\mathbb{R}^d)$ is defined by \eqref{eq_S}	where $\F$ denotes the Fourier transform
	\begin{align}\label{eq_cft}
	\F f(\xi) = \int_{\mathbb{R}^d} f(x) e^{-2\pi i x \xi} \,dx.
	\end{align}
		A set $E \subseteq \mathbb{R}^d$ is said to have a \emph{maximally Ahlfors regular boundary} if
	there exists a constant $\kappa_{\partial \ee}>0$ such that
	\begin{align*}
	\mathcal{H}^{d-1}\big(\partial E \cap B_{r}(x) \big)\geq \kappa_{\partial \ee} \cdot r^{d-1}, \qquad 0 < r \leq \mathcal{H}^{d-1}(\partial E)^{1/(d-1)}, \quad x \in \partial E.
	\end{align*}
	Here, $\mathcal{H}^{d-1}$ denotes the $(d-1)$-dimensional Hausdorff measure. The term maximal in the definition refers to the range of $r$ for which the estimate is required to hold. See Section \ref{sec:pre} for more context on Ahlfors regularity. In what follows, we denote for short $|\partial E| = \mathcal{H}^{d-1}\big(\partial E)$.

In this article we prove the following.
	
	\begin{theorem}\label{th1}
Let $\ee,F\subseteq\R^d$, $d\geq 2$, be compact domains with maximally Ahlfors regular boundaries with constants $\kappa_{\partial \ee},\kappa_{\partial F}$ respectively,
and assume that that $|\partial\ee||\partial F|\ge 1$. Consider the concentration operator \eqref{eq_S} and its eigenvalues $\{\lambda_n: n \in \mathbb{N}\}$.

Then for every $\alpha\in(0,1/2)$, there exists $A_{\alpha,d}\geq 1$ such that  for $\varepsilon\in(0,1/2)$:
\begin{multline}\label{eq_taa}
	\#\big\{n\in\N:\ \lambda_n \in(\varepsilon,1-\varepsilon)\big\}
	\leq A_{\alpha,d}  \cdot \frac{|\partial \ee|}{\kappa_{\partial \ee}  } \cdot \frac{|\partial F|}{\kappa_{\partial F} } \cdot \log\left( \frac{|\partial \ee||\partial F|}{ \kappa_{\partial \ee}\ \varepsilon}\right)^{2d(1+\alpha)+1}.
\end{multline}	
	\end{theorem}
The strength of Theorem \ref{th1} lies in the fact that the right-hand side of \eqref{eq_taa} grows only mildly on $\varepsilon$, in agreement with the Landau-Widom
asymptotic formula for one dimensional intervals \eqref{eqc}. In contrast, cruder estimates based on computing first and second moments of concentration operators,
as done often in sampling theory \cite{la67}, lead to error bounds of the order $O(1/\varepsilon)$.

A result closely related to Theorem \ref{th1} is presented in the recent article \cite{isma23}. For $F=[0,1]^d$ and $E=r K$, where $r\geq1$ is a dilation parameter and $K \subset B_1(0) \subset \mathbb{R}^d$ is a \emph{convex, coordinate symmetric domain} \cite[Theorem 1.1]{isma23} gives the following bound for $\varepsilon \in (0,1/2)$:
\begin{align}\label{eq_d}
\#\big\{n\in\N:\ \lambda_n \in(\varepsilon,1-\varepsilon)\big\}
\leq C_d \cdot \max\{ r^{d-1} \log(r/\varepsilon)^{5/2}, \log(r/\varepsilon)^{5d/2}\}.
\end{align}
For large $r$, the right-hand side of \eqref{eq_d} becomes
$O_{d}\big(r^{d-1} \log(r/\varepsilon)^{5/2})$ while Theorem \ref{th1} gives the weaker bound $O_{\alpha,d}\big(r^{d-1} \log(r/\varepsilon)^{2d(\alpha+1)+1}\big)$, or, at best, the slightly better $O_{\alpha,d}\big(r^{d-1} \log(r/\varepsilon)^{2d(\alpha+1)}\big)$ by applying other technical results presented below.\footnote{Indeed, the bound $O_{\alpha,d}\big(r^{d-1} \log(r/\varepsilon)^{2d(\alpha+1)}\big)$ follows from Theorem \ref{th:cube}, presented below, which is applicable to the domains in question, together with the discretization argument in the proof of Theorem \ref{th1}.} On the other hand, Theorem \ref{th1} applies to possibly non-convex, non-coordinate-symmetric and non-dilated domains $E$, and other regular domains $F$ besides cubes.

When $E$ and $F$ are both cubes, even slightly stronger estimates hold, as follows by tensoring sharp bounds for one dimensional intervals (that match the Landau-Widom asymptotic \eqref{eqc}); see, e.g. \cite[Theorem 1.4]{isma23}. This can be used to argue that in \eqref{eq_taa} the factors ${|\partial \ee|}$ and ${|\partial F|}$ cannot be replaced by smaller powers. Similarly, while the power of the logarithm in \eqref{eq_taa} can conceivably be reduced, the logarithmic factor cannot be completely removed.

Our work is in great part motivated by applications where concentration domains may be non-convex, such as the complement of a disk within a two dimensional square; see Section \ref{sec_sig}. Such a domain $E$ is allowed by Theorem \ref{th1} (and Theorems \ref{th2} and \ref{th3} below) and has moreover a favorable regularity constant $\kappa_{\partial E}$.

\subsection{Discretization of continuous domains}
	Theorem \ref{th1} is obtained by taking a limit on a more precise result concerning a discrete setting, which is our main focus.
	
	We consider a \emph{resolution parameter} $L>0$ and define the \emph{discrete Fourier transform} $\mathcal{F}_L:  L^2((-L/2,L/2)^d)\to \ell^2(L^{-1}\Z^d)$ by
	\begin{align}\label{eq_dft}
		\mathcal{F}_L f(k/L)=\int_{(-L/2,L/2)^d} f(x) e^{-2\pi i x k/L} dx, \qquad k \in \mathbb{Z}^d.
	\end{align}
	We think of $L$ as a discretization parameter for an underlying continuous problem. 
	
	Let us define the \emph{discretization at resolution} $L> 0$ of a domain $\ee\subseteq\R^d$ by
	\begin{align}\label{eq_eL}
	\ee_L=L^{-1}\Z^d\cap \ee.
	\end{align}
	Given two compact domains $\ee \subseteq \mathbb{R}^d$ and $F\subseteq (-L/2,L/2)^d$, consider the
	\emph{discretized concentration operator} $T:L^2(F)\to L^2(F)$ given by
	\begin{align}\label{eq_T2}
	T=\chi_F \mathcal{F}_L^{-1}\chi_{\ee_L}\mathcal{F}_L.
	\end{align}
		Our second result reads as follows.
		
	\begin{theorem}\label{th2} 
	Let $\ee,F\subseteq\R^d$, $d\geq 2$, be compact domains with maximally Ahlfors regular boundaries with constants $\kappa_{\partial \ee},\kappa_{\partial F}$ respectively, and assume that that $|\partial\ee||\partial F|\ge 1$. 
	
	Fix a discretization resolution $L\geq |\partial\ee|^{-1/(d-1)}$ such that $F\subseteq (-L/2,L/2)^d$ and consider the discretized concentration operator \eqref{eq_T2} and its eigenvalues $\{\lambda_n: n \in \mathbb{N}\}$.
		
	Then for every $\alpha\in(0,1/2)$ there exists $A_{\alpha,d}\geq 1$ such that for $\varepsilon\in(0,1/2)$:
		\begin{multline}\label{eq_i7}
			\#\big\{n\in\N:\ \lambda_n\in(\varepsilon,1-\varepsilon)\big\}
			\leq A_{\alpha,d}  \cdot \frac{|\partial \ee|}{\kappa_{\partial \ee}  } \cdot \frac{|\partial F|}{\kappa_{\partial F} } \cdot \log\left( \frac{|\partial \ee||\partial F|}{ \kappa_{\partial \ee}\ \varepsilon}\right)^{2d(1+\alpha)+1}.
		\end{multline}
	\end{theorem}

	The eigenvalue sequence on the left-hand side of \eqref{eq_i7} is finitely supported (with a bound that depends on the resolution parameter $L$). In contrast, the right-hand side of \eqref{eq_i7} is independent of $L$. In applications, this helps capture the transition between analog models and their finite computational counterparts, rigorously showing that the latter remain faithful to the former.
	
	\subsection{The discrete Fourier transform}
	Finally, we consider a discrete concentration problem associated with the usual \emph{discrete Fourier transform}, denoted
	\[\mathcal{F}_1:  L^2((-1/2,1/2)^d)\to \ell^2(\Z^d)\]
	for consistency with \eqref{eq_dft}.
	
	Given a finite set $\Omega \subseteq \Z^d$ and $F \subseteq (-1/2,1/2)^d$, the \emph{discrete Fourier concentration operator} $T:L^2(F)\to L^2(F)$ is defined as
	\begin{align}\label{eq_T}
		T = \chi_F \mathcal{F}_1^{-1}\chi_{\Omega}\mathcal{F}_1 .
	\end{align}
	The discrete boundary of a set $\Omega\subseteq \Z^d$ is given by
	\begin{align}\label{eq_b}
		\partial \Omega = \{k \in \Omega: \min\{|j-k|: j \in \mathbb{Z}^d \smallsetminus \Omega\} = 1\}.
	\end{align}
	
	We say that $\Omega\subseteq \Z^d$ has a \emph{maximally Ahlfors regular boundary} if there exists a constant $\kappa_{\partial \Omega}$ such that 
	\begin{align*}
	\inf_{k\in \partial \Omega} \#\big(\partial \Omega\cap k+ [-n/2,n/2)^d\big)\geq \kappa_{\partial \Omega} \cdot n^{d-1},\quad 1\leq n\leq (\#\partial \Omega)^{1/(d-1)},\ k\in \partial \Omega.
	\end{align*}
	(Note the slight notational abuse: though $\Omega \subseteq \mathbb{Z}^d \subseteq \mathbb{R}^d$, the notions of boundary and boundary regularity are to be understood in the discrete sense.)
	
	Our last result reads as follows.
	
	\begin{theorem}\label{th3}
		Let $d\ge 2$, $\Omega\subseteq \Z^d$ a finite set with {maximally Ahlfors regular boundary} and constant $\kappa_{\partial \Omega}$. Let $F\subseteq (-1/2,1/2)^d$ be compact with maximally Ahlfors regular boundary and constant $\kappa_{\partial F}$. Assume that $\#\partial\Omega \cdot |\partial F|\ge 1$, and consider the concentration operator \eqref{eq_T} and its eigenvalues $\{\lambda_n: n \in \mathbb{N}\}$.
				
		Then for every $\alpha\in(0,1/2)$ there exists $A_{\alpha,d}\geq 1$ such that for $\varepsilon\in(0,1/2)$:
		\begin{multline*}
			\#\big\{n\in\N:\ \lambda_n\in(\varepsilon,1-\varepsilon)\big\}
			\leq A_{\alpha,d} \cdot \frac{\#\partial \Omega}{\kappa_{\partial \Omega}} \cdot \frac{|\partial F|}{\kappa_{\partial F} }    \cdot \log\left( \frac{\#\partial \Omega \cdot |\partial F|}{ \kappa_{\partial \Omega}\ \varepsilon}\right)^{2d(1+\alpha)+1}.
		\end{multline*}
	\end{theorem}

	\subsection{One sided estimates}
	We remark that bounds on the number of intermediate eigenvalues, as in Theorems \ref{th1}, \ref{th2} and \ref{th3}, can be equivalently formulated in terms of the distribution function
	\begin{align*}
		N_\varepsilon := \{ n \in \mathbb{N}: \lambda_n > \varepsilon\}, \qquad \varepsilon\in(0,1).
	\end{align*}
	\begin{rem}\label{rem_N}
	For example, for $\varepsilon\in(0,1)$ under the assumptions of Theorem \ref{th1} we have
	\begin{align}\label{eq_N}
		\big|\#N_\varepsilon(S)-|E|\cdot |F|\big|\leq C_{\alpha,d} \cdot \frac{|\partial \ee|}{\kappa_{\partial \ee}  } \cdot \frac{|\partial F|}{\kappa_{\partial F} } \cdot \log\left( \frac{|\partial \ee||\partial F|}{ \kappa_{\partial \ee}\ \min\{\varepsilon, 1-\varepsilon\}}\right)^{2d(1+\alpha)+1}.
	\end{align}
	\end{rem}
	See Section \ref{sec_rem} for details.
	
	\subsection{Significance}\label{sec_sig}
	Fourier concentration operators arise in problems where functions are assumed to be supported on a certain domain $F$ and their Fourier transforms are known or measured on a second domain $E$. The insight that the class of such functions is approximately a vector space of dimension $|E| \cdot |F|$ is at the core of many classical and modern physical and signal models, and measurement and estimation methods \cite{MR2883827} \cite[``A Historical View'']{grunbaum2022serendipity}. 
	
	While in classical applications, such as telecommunications, the concentration domains are rectangles or unions thereof, the increasingly complex geometric nature of data and physical models has sparked great interest in spatial and frequency concentration domains with possibly intricate shapes. To name a few: (a) In geophysics and astronomy the power sprectrum of various quantities of interest is often assumed to be bandlimited to a disk or annulus and needs to be estimated from measurements taken on a domain as irregular as a geological continent \cite{Simons+2003a,Han+2008a,MR2812375,simons2000isostatic}; (b) The Fourier extension algorithm approximates a function on an arbitrary domain by a Fourier series on an enclosing box and crucially exploits the expected moderate size of the plunge region \eqref{eq_Ms} \cite{MR3803283,MR3989238}; (c) Noise statistics are often estimated from those pixels of a square image located outside a central disk, which is assumed to contain the signal of interest --- thus, the need to sample pure noise leads one to consider the complement of a disk within a two dimensional square as concentration domain (or, more realistically, a set of grid points within that domain) \cite{bhzhsi16, MR3634987, ansi17, anro20}.
	
	The expected low dimensionality of physical and signal models based on spatio-temporal constraints is often exploited without direct computation of eigendecompositions of Fourier concentration operators (which may in fact be ill-posed in the absence of symmetries). Rather, the expected asymptotic spectral profile of such operators informs strategies based on randomized linear algebra and information theory. The quest to analyze such models and methods has motivated a great amount of recent research \cite{is15,KaRoDa,MR3887338,MR3926960,BoJaKa,Os, isma23} which led to deep and far reaching improvements over more classical non-asymptotic results \cite{MR169015}. However, the mentioned literature covers only rectangular or other convex and symmetric domains, which precludes applications involving complex geometries. Due to this limitation, signal and measurement models with a complex geometric nature are often analyzed based on estimates much cruder than \eqref{eq_taa}, which have error factors of the order $1/\varepsilon$ instead of $\log(1/\varepsilon)^\alpha$ \cite{MR2812375,MR3803283, anro20}, and thus poorly reflect the remarkable practical effectiveness of low dimensional models. Estimates on measurement/reconstruction complexity, estimation confidence, or approximation/stability trade-offs are as a consequence orders of magnitude too conservative. Our work bridges this theory to practice gap by providing the first spectral deviation estimates for Fourier concentration operators valid without simplifying symmetry or convexity assumptions that also match the precision of what is rigorously known for intervals \cite{MR140732,sle78, is15,KaRoDa,MR3887338,MR3926960,BoJaKa,Os, isma23}. In addition, while Euclidean Fourier concentration operators help analyze
	computational schemes in their asymptotic continuous limit, our results for the discrete Fourier transform apply, more quantitatively, to finite settings, as they occur in many applications.
	
	\subsection{Methods and related literature}\label{sec_r}
	
	We work for the most part with the discrete Fourier transform and then obtain consequences for the continuous one by a limiting argument. Theorem \ref{th2} is thus a more precise and quantitative version of Theorem \ref{th1}, and is proved in two steps. We first revisit Israel's argument \cite{is15} and adapt it to prove eigenvalue estimates when one of the domains is a rectangle and the other one is a general multi-dimensional discrete domain (Theorem \ref{th:cube} below). These estimates are slightly stronger than those in Theorem \ref{th2}, and the extra precision is exploited in the subsequent step. We follow the method of almost diagonalization with wave-packets, which we achieve, unlike \cite{is15}, through a redundant system (frame) instead of an orthonormal basis. The versatility of redundant expansions helps us avoid requiring symmetries from the concentration domain.
	
	The second step is a decomposition, rescaling, and dyadic approximation argument, implemented by means of $p$-Schatten norm estimates for certain Hankel operators, and especially by quantifying those estimates as a function of $p$, as $p \to 0^+$. In this way we reduce the problem to the case where one of the domains is a rectangle, while relying on the refined estimates derived in the first step.
	
	Our intermediate result, Theorem \ref{th:cube}, is close in spirit to Theorem 1.1 in \cite{isma23} (which appeared as we were finishing this article). The estimates in \cite{isma23}, formulated in the context of the continuous Fourier transform and concerning dilated convex domains, are stronger than what follows from Theorem \ref{th:cube} in that regime, as \cite[Theorem 1.1]{isma23} involves smaller powers of a certain logarithmic factor (see also Section \ref{sec_e} and \eqref{eq_d}). On the other hand, Theorem \ref{th:cube} concerns a sufficiently regular possibly non-convex and non-symmetric domain, and covers the discrete Fourier transform (while Theorem \ref{th2} concerns two such domains).
	
	We also mention our recent work on concentration operators for the short-time Fourier transform \cite{maro21}, that also makes use of Ahlfors regularity and Schatten norm estimates. Though the goals and results are philosophically similar to those in the present article, the settings are rather different from the technical point of view. Indeed, the arguments used in \cite{maro21} rely on the rapid off-diagonal decay of the reproducing kernel of the range of the short-time Fourier transform, and do not seem to be applicable to Fourier concentration operators.
	
	Finally, we comment on the notion of maximal Ahlfors regularity introduced in this article, which is a variant of more common notions (see Section \ref{sec_br}). While this condition is sufficient to treat all the applications that we are aware of, it is probably non-optimal. Indeed, the eigenvalue estimate \eqref{eq_taa} is robust under certain geometric operations on the concentration domains that may not preserve Ahlfors regularity (see for example the dyadic decomposition and approximation in Section \ref{sec_dec}). However, we do not presently know how to formulate a simple and elegant regularity assumption that leads to a significant refinement of our results. 
	
	The remainder of the article is organized as follows. Section \ref{sec:pre} sets up the notation and provides background on boundary regularity. Section \ref{sec_i} revisits and aptly adapts the technique from \cite{is15}. This is used in Section~\ref{sec_gc} to prove Theorem \ref{th:cube}. Theorem \ref{th2} is proved in Section \ref{sec_step2}, Theorem \ref{th1} is proved in Section \ref{sec:con}, and Theorem \ref{th3} is proved in Section \ref{sec_dis}. Remark \ref{rem_N} is proved in Section \ref{sec_rem}.
	
	\section{Preliminaries}\label{sec:pre}
	
	\subsection{Notation}
	We shall focus on Theorem \ref{th2} and set up the notation accordingly. Theorems \ref{th1} and \ref{th3} will be obtained afterwards as an application of Theorem \ref{th2}.
	
	We denote cubes by $Q_a=[-a/2,a/2)^d$. The Euclidean norm on $\mathbb{R}^d$ is denoted $|\cdot|$.
	For two non-negative functions $f,g$ we write $f \lesssim g$ if there exist a constant $C$ such that $f(x) \leq C g(x)$, and write $f \asymp g$ if $f \lesssim g$ and $g \lesssim f$. The implied constant is allowed to depend on the dimension $d$ and the parameter $\alpha$ from Theorems \ref{th1}, \ref{th2} and \ref{th3}, but not on other parameters.
	
	We enumerate the eigenvalues of a compact self adjoint operator $S: \H \to \H$ acting on a Hilbert space $\H$
	as follows:
	\begin{align}
	\label{eqord}
		\lambda_k=\inf\{\|S-S_0\| : \ S_0\in \mathcal{L}(L^2(\R^d)),\,\dim(\mathrm{Range} (S_0))<k\},
	\qquad  k \geq 1.
	\end{align}
	Then $\{ \lambda_k: k \geq 1 \} \smallsetminus \{0\} = \sigma(S) \smallsetminus \{0\}$ as sets with multiplicities --- see, e.g., \cite[Lemma 4.3]{dj}. 
	
	Recall that the \emph{discretization at resolution} $L> 0$ of a set $E \subseteq \mathbb{R}^d$ is defined by \eqref{eq_eL}. We also write \[\ee_L^c= L^{-1}\Z^d\smallsetminus\ee_L\] and $\partial\ee_L$ for the points in $\ee_L$ which are at distance $L^{-1}$ of $\ee_L^c$:
	\begin{align*}
	\partial\ee_L = \big\{k/L \in E_L: \min\{|k/L-j/L|: j/L \in \ee_L^c\} = L^{-1}\big\}.
	\end{align*}
	For $L=1$ this is consistent with \eqref{eq_b}.
		
	We will work with the discrete Fourier transform $\mathcal{F}_L:  L^2((-L/2,L/2)^d)\to \ell^2(L^{-1}\Z^d)$ given by \eqref{eq_dft} and reserve the notation $\mathcal{F}f$ or $\widehat{f}$ for the continuous Fourier transform \eqref{eq_cft}. Note that if $\supp (f)\subseteq (-L/2,L/2)^d$, then $\mathcal{F}_L f(k/L)=\mathcal{F} f(k/L)$ for every $k\in \Z^d$.

	We also write $P_{\ee,L}=\mathcal{F}_L^{-1}\chi_{\ee_L}\mathcal{F}_L.$ For $F\subseteq (-L/2,L/2)^d$ we define the operator $T=T_{\ee,F,L}:L^2(F)\to L^2(F)$ by 
	\begin{align*}
	T=T_{\ee,F,L}=\chi_F P_{\ee,L}
	\end{align*}
	and let $\lambda_n=\lambda_n(T)$ denote its eigenvalues as in \eqref{eqord}. An easy computation shows that
	\[T_{t^{-1}\ee,tF,tL}=\mathcal{M}_{t^{-1}}T_{\ee,F,L}\mathcal{M}_t, \qquad t>0,\]
	where $\mathcal{M}_t$ denotes the \emph{dilation operator}
	\begin{align*}
	\mathcal{M}_t f(x)=f(tx).
	\end{align*}
	In particular,
	\begin{align}
		\label{eq:stretch}
		\lambda_n(T_{t^{-1}\ee,tF,tL})=\lambda_n(T_{\ee,F,L}), \qquad n\in \N.
	\end{align}
	
	\subsection{Boundary regularity}\label{sec_br}
	Let us introduce regularity of sets in more generality and discuss a few properties.
	
	An $\mathcal{H}^{d-1}$-measurable set $\xx \subseteq \R^d$ is said to be {\em lower Ahlfors $(d-1)$-regular} (regular for short) at scale $\eta_\xx>0$ if there exists a constant $\kappa_\xx>0$ such that
	\begin{align*}
		\mathcal{H}^{d-1}\big(\xx \cap B_{r}(x) \big)\geq \kappa_\xx \cdot r^{d-1}, \qquad 0 < r \leq \eta_\xx, \quad x \in \xx.
	\end{align*}
	(See for example \cite[Definitions 2.1 and 2.2]{MR4510936}).
	In the literature, Ahlfors regularity is usually stated for $\eta_X\sim \text{diam}(X)$. In contrast, we work with various scales and introduce the term {\em maximally Ahlfors regular} whenever the regularity condition holds at scale $\eta_X=\mathcal{H}^{d-1}(\xx )^{1/(d-1)}$. The use of this scale instead of the usual $\text{diam}(X)$ allows us to include disconnected sets with distant connected components in our analysis.
	
	Note that if $\xx \subseteq \R^d$ is regular at scale $\eta_\xx>0$ with constant $\kappa_\xx>0$ and $t>0$, then $t \xx \subseteq \R^d$ is regular at scale $\eta_{t\xx}=t \eta_\xx$ with constant $\kappa_{t\xx}=\kappa_\xx$.
	By differentiation around a point of positive $\mathcal{H}^{d-1}$-density, 
	\begin{align}\label{eqkappa}
	\kappa_\xx\le c_d,
	\end{align}
	for any regular $\xx$ of finite $\mathcal{H}^{d-1}$-measure. We also mention that if $X$ is regular with parameters $\eta_X$ and $\kappa_X$, then choosing an arbitrary $x\in X$ gives
	\begin{align}\label{eqone}
	\mathcal{H}^{d-1}\big(\xx  \big)\ge \mathcal{H}^{d-1}\big(\xx \cap B_{\eta_X}(x) \big)\geq \kappa_\xx \cdot \eta_X^{d-1}.
	\end{align}
	
	We shall use the following basic result, derived from \cite{Ca}.
	
	\begin{lemma}\label{lem:coar}
		There exists a universal constant $C_d>0$ such that for every compact set $\xx \subseteq \mathbb{R}^{d}$ that is regular at scale $\eta_\xx>0$ with constant $\kappa_\xx$ and every $s>0$,
		\begin{align*}
			|\xx+B_s(0)| \leq  \frac{C_d}{\kappa_\xx}  \cdot \mathcal{H}^{d-1}(\xx) \cdot s \cdot \Big(1+\frac{s^{d-1}}{\eta^{d-1}_\xx}\Big).
		\end{align*}
	\end{lemma}
	\begin{proof}
		From \cite[Theorems 5 and 6]{Ca} it follows that  
		\begin{align*}
			\mathcal{H}^{d-1} \big(\{x \in \mathbb{R}^{d}:\ d(x,\xx)=r\} \big) \leq  \frac{C_d}{\kappa_\xx}  \cdot \mathcal{H}^{d-1}(\xx) \cdot \Big(1+\frac{r^{d-1}}{\eta_\xx^{d-1}}\Big),
		\end{align*}
		for almost every $r>0$, and in addition, $| \nabla d(x,\xx) | = 1$, for almost every $x \in \mathbb{R}^{d}$. From this and the coarea formula --- see, e.g., \cite[Theorem 3.11]{evga92} ---
		it follows that
		\begin{align*}
			|\xx+B_s(0)|&= \int_{\R^d} \chi_{[0,s)}(d(x,\xx)) dx 
			= \int_{\R^d} \chi_{[0,s)}(d(x,\xx)) | \nabla d(x,\xx) | dx 
			\\ &= \int_0^s \mathcal{H}^{d-1}\big(\{x:\ d(x,\xx)=r\} \big) dr
			\leq  \frac{C_d}{\kappa_\xx} \mathcal{H}^{d-1}(\xx)  \int_0^s  \Big(1+\frac{r^{d-1}}{\eta_\xx^{d-1}}\Big) dr
			\\ &\leq  \frac{C_d}{\kappa_\xx}  \mathcal{H}^{d-1}(\xx) s\Big(1+\frac{s^{d-1}}{\eta_\xx^{d-1}}\Big). \qedhere
		\end{align*}
	\end{proof}

	\begin{corollary}\label{cor:coar}
		For $\ee\subseteq\R^d$ a compact domain with regular boundary at scale $\eta_{\partial\ee}\ge 1$ with constant $\kappa_{\partial\ee}$ and a discretization resolution $L\ge 1$, we have
		\begin{align*}
			L^{-d}\#\ee_L\lesssim |\ee|+\frac{|\partial\ee|}{\kappa_{\partial\ee} L}.
		\end{align*}
		In particular, for $d\geq 2$,
		\begin{align*}	 	L^{-d}\#\ee_L\lesssim  \frac{\max\{|\partial\ee|^{d/(d-1)},1\}}{\kappa_{\partial\ee}}.	 
		\end{align*}
	\end{corollary}
	\begin{proof}
		Recall that $Q_{L^{-1}}=L^{-1}[-1/2,1/2)^d$ and define
		$\ee_L'=\{m\in \ee_L :\ m+Q_{L^{-1}}\subseteq \ee\}$.
		From Lemma~\ref{lem:coar}, we get
		\begin{align*}
			L^{-d}\#\ee_L &= 	\Big|\bigcup_{m\in\ee_L'} m+Q_{L^{-1}}\Big|+	\Big|\bigcup_{m\in\ee_L\smallsetminus\ee_L'} m+Q_{L^{-1}}\Big|\le |\ee|+|\partial\ee+B_{L^{-1}\sqrt{d}}(0)|
			\\ &\lesssim |\ee|+ \frac{|\partial\ee|}{\kappa_{\partial\ee} L}.
		\end{align*}
		Finally, the second inequality follows from the isoperimetric inequality $|\ee|\lesssim|\partial\ee|^{d/(d-1)}$ and \eqref{eqkappa}.
	\end{proof}
	
	\section{Israel's argument revisited}\label{sec_i}
	
	We now revisit the core argument of \cite{is15} and aptly adapt it so as to treat multi-dimensional and discrete domains.
		
	\subsection{Israel's lemma}
	
	We need a slight generalization of Lemma~1 in \cite{is15}, phrasing it in terms of frames rather than orthonormal bases. We include a proof for the sake of completeness.
	
	Recall that a frame for a Hilbert space $\H$ is a subset of vectors 	$\{\phi_i\}_{i\in\mathcal{I}}$ for which the exist constants $0<A,B<\infty$ ---
	called lower and upper frame bounds ---	such that
	\begin{align*}
		A \| f \|^2 \leq \sum_{i \in \mathcal{I}} |\langle f, \phi_i \rangle|^2 \leq B \|f\|^2, \qquad f \in \H.
	\end{align*}
	If, moreover, $A=B$, the say that the frame is \emph{tight}.
	
	\begin{lemma}\label{lem:main}
		Let $T:\H\to \H$ be a positive, compact, self-adjoint operator 
		on a Hilbert space $\H$ with $\|T\|\leq 1$ and eigendecomposition 
		$T=\sum_{n\ge 1}\lambda_n\langle \cdot ,f_n\rangle f_n$.
		
		Let $\{\phi_i\}_{i\in\mathcal{I}}$ be a frame of unit norm vectors for $\H$ with lower frame bound $A$.		If $\mathcal{I}= \mathcal{I}_1\cup\mathcal{I}_2\cup\mathcal{I}_3$, and  
		\begin{equation}\label{eq:eps-cubed}
			\sum_{i\in\mathcal{I}_1}\|T\phi_i\|^2+\sum_{i\in\I_3}\|(I-T)\phi_i\|^2\leq \frac{A}{2}\varepsilon^2,
		\end{equation}
		then $\# M_\varepsilon(T)\leq \frac{2}{A}\#\I_2,$ where $M_\varepsilon(T)$ is defined as in \eqref{eq_Ms}.
	\end{lemma}
	\proof Let $S_\varepsilon=\text{span}\{f_n:\ \lambda_n\in(\varepsilon,1-\varepsilon)\}$ and let $P_\varepsilon:\H\to S_\varepsilon$ denote the orthogonal projection onto $S_\varepsilon$. Observe that 
	$$
	\varepsilon\|f\|\leq \|Tf\|,\quad\text{and}\quad\varepsilon\|f\|\leq \|(I-T)f\|,\quad f\in S_\varepsilon,
	$$
	where the second inequality follows from the fact that for $f\in S_\varepsilon$ one has $\|f\|-\|(I-T)f\|\leq  \|Tf\|\leq (1-\varepsilon)\|f\|$.
	Also note that $T$ and $P_\varepsilon$ commute since $S_\varepsilon$ is spanned by a collection of eigenvectors of $T$.
	Therefore, by \eqref{eq:eps-cubed} we obtain
	$$
	\sum_{i\in\I_1\cup\I_3}\varepsilon^2\|P_\varepsilon \phi_i\|^2\leq \sum_{i\in\I_1}\|TP_\varepsilon \phi_i\|^2+\sum_{i\in \I_3}\|(I-T)P_\varepsilon \phi_i\|^2\leq \frac{A}{2}\varepsilon^2,
	$$
	which implies 
	\begin{equation}\label{eq:A/2}
		\sum_{i\in\I_1\cup\I_3}\|P_\varepsilon \phi_i\|^2\leq \frac{A}{2}.
	\end{equation}
	Using the frame property we get for  $f \in S_\varepsilon$: 
	$$
	A\|f\|^2\leq \sum_{i\in\I}|\langle f,\phi_i\rangle |^2=\sum_{i\in\I}|\langle f,P_\varepsilon \phi_i\rangle|^2.
	$$
	Now assume that $\text{dim}(S_\varepsilon)\geq 1$ (otherwise the result is trivial), take an orthonormal basis $\{\psi_k\}_{k=1}^{\text{dim}(S_\varepsilon)}$ of $S_\varepsilon$, and sum the inequality above over all basis elements to derive
	\begin{align*}
		A \cdot \# M_\varepsilon(T)&=A \cdot \text{dim}(S_\varepsilon)= A\sum_{k=1}^{\text{dim}(S_\varepsilon)}\|\psi_k\|^2\leq \sum_{k=1}^{\text{dim}(S_\varepsilon)}\sum_{i\in\I}|\langle \psi_k,\phi_i\rangle|^2
		\\
		&= \sum_{i\in\I}\|P_\varepsilon\phi_i\|^2\leq \sum_{i\in\I_2}\|\phi_i\|^2+\frac{A}{2}= \#\I_2 +\frac{A}{2}\le \#\I_2 +\frac{A}{2}\# M_\varepsilon(T),
	\end{align*}
	where in the second line we used \eqref{eq:A/2}.
	This shows that  $\# M_\varepsilon(T)\leq \frac{2}{A}\#\I_2.$ \pbox
	
	\subsection{Local trigonometric frames}
	
	In this section, we construct a tight frame that allows us to apply Lemma~\ref{lem:main}.
	
	Let $\alpha>0$, and $\theta\in\mathcal{C}^\infty(\R)$ be such that 
	\begin{enumerate}[label=(\roman*)]
		\item\label{theti} $\theta(x)=1,$ for $x\geq 1$, and $\theta(x)= 0,$ for $x\leq -1$,
		\item\label{thetii} $\theta(-x)^2+\theta(x)^2=1$, for every $x\in\R$,
		\item\label{thetiii}  $|D^k\theta(x)|\leq C_\alpha^k k^{(1+\alpha)k}$, for all $k\in\N_0$, all $x\in\R$, and a constant $C_\alpha>0$.
	\end{enumerate}
	See, for example, \cite[Proposition~1]{is15} or \cite[Chapter 1]{MR1408902} for the existence of such a function.
	
	Let $W>0$. We decompose the interval $\big(\hspace{-3pt}-\frac{W}{2},\frac{W}{2}\big)$ into disjoint intervals 
	$$
	I_j=x_j+\frac{W}{3 \cdot 2^{|j|+1}}[-1,1),\quad j\in\Z,
	$$
	where 
	\[x_j=\frac{\sign(j) W}{2}\Big(1-\frac{1}{2^{|j|}}\Big).\]
	Note that $|I_j|=|I_{|j|}|=2|I_{|j|+1}|$ for every $j\in\Z$. We will also denote $D_j=I_j\cup I_{j+1}$.
	Now define 
	$$
	\theta_j(x)=\theta\left(\frac{2(x-x_{j})}{|I_{j}|}\right)\theta\left(-\frac{2(x-x_{j+1})}{|I_{j+1}|}\right).
	$$
	We have that  $\theta_j(x)=0$ for $x \notin D_j $, and furthermore by properties \ref{theti} and \ref{thetii} 
	\begin{align*}
		\|\theta_j\|_2^2&= \int_{I_j}\theta\left(\frac{2(x-x_{j})}{|I_{j}|}\right)^2 dx+\int_{I_{j+1}}\theta\left(-\frac{2(x-x_{j+1})}{|I_{j+1}|}\right)^2 dx
		\\&= \frac{|I_{j}|}{2} \int_{-1}^1 \theta(x)^2 dx + \frac{|I_{j+1}|}{2} \int_{-1}^1 \theta(x)^2 dx=\frac{|D_{j}|}{2}.
	\end{align*}
	We define the set of vectors
	\begin{equation}\label{def:Phi}
		\phi_{j,k}(x)=\sqrt{\frac{2}{|D_{j}|}}\cdot \theta_j(x)\cdot\exp\left(2\pi i\frac{xk}{|D_j|}\right),\quad j,k\in\Z,
	\end{equation}
	and note that $\|\phi_{j,k}\|_2=1$.
	
	\begin{lemma}
		The family $\{\phi_{j,k}\}_{j,k\in\Z}$ defined in \eqref{def:Phi} forms a tight frame for $L^2(-W/2,W/2)$ with frame constants $A=B=2$.
	\end{lemma}
	
	\begin{proof}
		Let $f \in L^2(-W/2,W/2)$ and set $f_j:=f \chi_{I_j}$ so that $f=\sum_{j\in \Z}f_j$. Since $\text{supp}(\theta_j)\subseteq D_j=I_j\cup I_{j+1}$, we  observe 
		\begin{align*}
			\sum_{j,k\in\Z}|\langle f,\phi_{j,k}\rangle|^2&=\sum_{j,k\in\Z}\left|\langle f_j+f_{j+1},\phi_{j,k}\rangle\right|^2.
		\end{align*}
		As $\left\{ |D_j|^{-1/2}\exp\left(2\pi ikx/|D_j|\right)\right\}_{k\in\Z}$ is an orthonormal basis for $L^2(D_j)$, we find that
		$$
		\sum_{k\in\Z}|\langle f_j+f_{j+1},\phi_{j,k}\rangle|^2=2 \| (f_j+f_{j+1})\theta_j \|^2_2=2\|f_j\theta_j \|^2_2+2\| f_{j+1}\theta_j \|^2_2.
		$$
		Combining both identities and using property \ref{thetii}, we conclude
		\begin{align*}
			\sum_{j,k\in\Z}|\langle f,\phi_{j,k}\rangle|^2&=2\sum_{j\in\Z}\Big(\|f_j\theta_j \|^2_2+\| f_{j+1}\theta_j \|^2_2\Big)
			=2 \sum_{j\in\Z}\Big(\|f_j\theta_j \|^2_2+\| f_j\theta_{j-1} \|^2_2\Big)
			\\ &= 2 \sum_{j\in\Z} \int_{I_j}|f(x)|^2 \left(\theta\left(\frac{2(x-x_{j})}{|I_{j}|}\right)^2+\theta\left(-\frac{2(x-x_{j})}{|I_{j}|}\right)^2\right) dx
			\\&=2 \sum_{j\in\Z}\|f_j \|^2_2=2\|f\|_2^2. \qedhere
		\end{align*}
	\end{proof}
	
	Let $0<W_i\le L,\ i=1,...,d,$  and consider the rectangle $\prod_{i=1}^d(-W_i/2,W_i/2)$. Set also
	\begin{align*}
	W_{\max} := \max_{i=1,...,d} W_i.
	\end{align*}
	We define a frame for $L^2\big(\prod_{i=1}^d(-W_i/2,W_i/2)\big)$ via the tensor product 
	\begin{equation*}
		\Phi_{j,k}(x)=\Phi_{j_1,...,j_d,k_1,...k_d}(x_1,...,x_d)=\phi_{j_1,k_1}(x_1)\cdot\ldots\cdot \phi_{j_d,k_d}(x_d),
	\end{equation*}
	where each family $\{\phi_{j_i,k_i}(x_i)\}_{j_i,k_k\in\Z}$ is the frame for $L^2(-W_i/2,W_i/2)$ given by \eqref{def:Phi}.
	This construction also yields a tight frame with frame bounds equal to $2^d$.
	
	\subsection{Energy estimates}
	Consider
	$$
	\psi_j(x)=\theta_j\Big(|D_j|x+x_j-\frac{W}{3 \cdot 2^{|j|+1}}\Big), \quad x\in \R,\ j\in\Z.
	$$ 
	A straightforward computation shows that $\psi_j$ is supported on $[ 0,1]$ and satisfies $|D^k \psi_j(x)|\leq \widetilde{C_\alpha}^k k^{(1+\alpha)k}$ by property \ref{thetiii}. As shown in \cite{MR1658223} or \cite[Lemma~4]{is15} it thus follows that $|\widehat{\psi_j}(\xi)|\leq A_\alpha \cdot \exp\left(-a_\alpha|\xi|^{(1+\alpha)^{-1}}\right)$. Since $1-\alpha\leq (1+\alpha)^{-1}$, we derive that $t^{(1+\alpha)^{-1}}\geq t^{1-\alpha}-1$ for $t\geq 0$. Adjusting the constant $A_\alpha$, we therefore get
	\begin{equation}\label{eq:F(theta)}
		|\widehat{\theta_j}(\xi)|\leq A_\alpha \cdot |D_j| \cdot \exp\left(-a_\alpha(|D_j|\cdot|\xi|)^{1-\alpha }\right),\quad \xi\in\R.
	\end{equation}
	With this at hand, we estimate the decay of
	$$
	\mathcal{F}(\Phi_{j,k})(\xi)=2^{d/2} \prod_{i=1}^d| D_{j_i}|^{-1/2}\cdot \widehat{\theta_{j_i}}\left(\xi_i-\frac{k_i}{| D_{j_i}|}\right),\quad \xi\in\R^d.
	$$
	Define
	\[M_j=\diag(|D_{j_1}|,...,|D_{j_d}|)\in\R^{d\times d}.\]
	 By \eqref{eq:F(theta)} (possibly enlarging $A_\alpha$) and $|\xi|^{1-\alpha}\leq \sum_{i=1}^d|\xi_i|^{1-\alpha}$, it follows 
	\begin{align}
		|\mathcal{F}(\Phi_{j,k})(\xi)|&\leq A_\alpha^d\prod_{i=1}^d   | D_{j_i}|^{1/2}\cdot \exp\left(-a_\alpha\big||D_{j_i}|\xi_i-k_i\big|^{1-\alpha }\right)\nonumber 
		\\
		&\le A_\alpha^d \cdot \text{det}(M_j)^{1/2}\cdot \exp\left(-a_\alpha\big|M_j(\xi-\xi_{j,k})\big|^{1-\alpha }\right)\label{eq:decay},
	\end{align}
	where $(\xi_{j,k})_i=k_i| D_{j_i}|^{-1}$.
	
	Consider now a compact domain $E \subseteq \mathbb{R}^d$.	
	Let $s\ge 1$ be a parameter that will be determined later.
	For $j\in\Z^d$ fixed, we cover the index set $\Z^d$ with three subsets as follows
	\begin{align}\label{eq_set1}
	\begin{aligned}
		\mathcal{L}_j^{\text{low}}&:=\big\{k\in\Z^d:\ \text{dist}(k,M_j\ee_L^c)\geq s  \big\};
		\\
		\mathcal{L}_j^{\text{med}}&:=\big\{k\in\Z^d:\ \text{dist}(k,M_j\ee_L)< s \text{, and } \text{dist}(k,M_j\ee_L^c)< s\};
		\\
		\mathcal{L}_j^{\text{high}}&:=\big\{k\in\Z^d:\ \text{dist}(k,M_j\ee_L)\geq s  \big\}.
	\end{aligned}
	\end{align}
	(Here, $\text{dist}$ is associated with the usual Euclidean distance.) We claim that
	\begin{align}
		\label{eq:lj}
		\mathcal{L}_j^{\text{med}}\subseteq\big\{k\in\Z^d:\ \text{dist}(k,M_j\partial\ee_L)< s\big\};
	\end{align}
	let us briefly sketch an argument.
	Fix $k\in\mathcal{L}_j^{\text{med}}$ and let $k_0\in M_j L^{-1}\Z^d$ minimize the distance to $k$. For any point $x\in M_j L^{-1}\Z^d$ with $|k-x|<s$ we can construct a path of adjacent points in $M_j L^{-1}\Z^d$ from $x$ to $k_0$ whose distance to $k$ is decreasing. In particular, choosing $x_1\in M_j\ee_L$ and $x_2\in M_j\ee_L^c$ at distance less than $s$ from $k$, we can connect $x_1$ and $x_2$ through a path of adjacent points in $M_jL^{-1}\Z^d$ that stays at distance less than $s$ from $k$. Necessarily, one of the points in the path must belong to  $M_j\partial\ee_L$, which proves \eqref{eq:lj}.
	
	The indices $(j,k)$ with $k\in \mathcal{L}_j^{\text{med}}$, and $j$ satisfying a certain condition specified below (see \eqref{eq_set2}) will play the role of $\mathcal{I}_2$ in Lemma \ref{lem:main}, so we need to estimate $\#(\mathcal{L}_j^{\emph{med}})$. 
		
	\begin{lemma}\label{lem:L}
		Let $\ee\subseteq \R^d$ be a compact domain with regular boundary at scale $\eta_{\partial\ee}\ge 1$ and constant $\kappa_{\partial\ee}$. 
	Let $L\ge W_{\max}$ and $s\ge 1$. Then for all $j\in\Z^d$ we have
	\begin{align*}
		\#\big\{k\in\Z^d:\ \emph{dist}(k,M_j\partial\ee_L)< s\big\}\lesssim \max\{W_{\max},1/\eta_{\partial\ee}\}^{d-1}\cdot  \frac{|\partial \ee|}{\kappa_{\partial\ee} } \cdot s^d.
	\end{align*}
\end{lemma}
\begin{proof}
	Since for every $x\in \R^d$ the cube $x+Q_1$ contains one point in $\Z^d$,
\begin{align}\label{eqdis}
		\#\big\{k\in\Z^d:\ \text{dist}(k,M_j\partial\ee_L)< s\}&\lesssim s^d\cdot\#\big\{k\in\Z^d:\  k\in M_j\partial\ee_L+Q_{1}\}.
	\end{align}
Second, we claim that for $0<a\leq 1$ and a set $X \subset \mathbb{R}$,
\begin{align}\label{eq_p}
\#\big\{k\in\Z:\    k\in aX+[-1/2,1/2)\}\le 2 \, \#\big\{k\in\Z:\   k\in  X+[-1/2,1/2)\}.
\end{align}
To see this, consider a map that sends $k=ax+t$, with $x \in X$ and $t \in [-1/2,1/2)$ to the unique integer in $x+[-1/2,1/2)$, denoted $k_*$ (the choice of $x$ and $t$ may be non-unique). Whenever $k_*=k'_*$, it follows that $|k-k'| \leq 1$. Hence, at most two $k's$ are mapped into the same $k_*$.

From \eqref{eqdis}, if we apply \eqref{eq_p} componentwise (noting that $(M_j)_{i,i}\leq W_{\max}$), we obtain for $s\geq 1$
	\begin{align*}
		\#\big\{k\in\Z^d:\ \text{dist}(k,M_j\partial\ee_L)< s\}&
		&\lesssim s^d\cdot\#\big\{k\in\Z^d:\ k\in W_{\max}\partial\ee_L+Q_{1} \}.
	\end{align*}

	Since for every $x\in \partial\ee_L$ there exists $x'\in \partial \ee$ such that $|x-x'|\leq L^{-1}$, and $W_{\max}/L\leq 1$, it follows that 
	\begin{align}
		\#\big\{k\in\Z^d:\ \text{dist}(k,M_j\partial\ee_L)< s\}&\lesssim s^d\cdot\#\big\{k\in\Z^d:\ k\in W_{\max}\partial\ee+Q_{3} \} \notag \\
		&=: s^d\cdot\# (\mathcal{K}_{W_{\max}}).\label{eq:pl}
	\end{align}
	Now let $k\in \mathcal{K}_{W_{\max}}$. There exists at least one point $x_k\in  \partial\ee $ such that $k\in W_{\max}x_k+Q_{3}$. In particular, we have that $x_k\in W_{\max}^{-1}k+Q_{4/W_{\max}}$. Therefore, for every $k\in \mathcal{K}_{W_{\max}}$ we get by regularity of $\partial\ee$
	\begin{align*}
		\kappa_{\partial\ee}\cdot \min\{W_{\max}^{-1},\eta_{\partial\ee}\}^{d-1} &\le \H^{d-1}\big( \partial \ee\cap B_{1/W_{\max}}(x_k)\big)
		\\
		&\le \H^{d-1}\big( \partial \ee\cap x_k+Q_{2 /W_{\max}} \big)
		\\ &\le \H^{d-1}\big( \partial \ee\cap W_{\max}^{-1}k+Q_{6 /W_{\max}} \big). \notag
	\end{align*}
So,
	\begin{align*}
		\kappa_{\partial\ee}\cdot \min\left\{W_{\max}^{-1},\eta_{\partial\ee}\right\}^{d-1}\cdot\#(\mathcal{K}_{W_{\max}})&\leq 
		\sum_{k\in\mathcal{K}_{W_{\max}}}  \H^{d-1}\big( \partial \ee\cap W_{\max}^{-1}k+Q_{6 /W_{\max}} \big)
		\\&\lesssim \sum_{k\in\Z^d}  \H^{d-1}\big( \partial \ee\cap W_{\max}^{-1}k+Q_{1 /W_{\max}} \big)\\ &=\H^{d-1}( \partial \ee). 
	\end{align*}
	Plugging this estimate into \eqref{eq:pl} completes the proof.
\end{proof}

Next, for a compact domain $E \subseteq \mathbb{R}^d$ and a parameter $s \geq 1$ we recall the sets \eqref{eq_set1}, introduce a second auxiliary parameter $0<\delta<1$, and define the following covering of $\Z^{2d}$: 
\begin{align}\label{eq_set2}
\begin{aligned}
	\Gamma^{\text{low}}&:=\big\{(j,k):\ \min_{i }| D_{j_i}|\geq \delta,\ k\in\mathcal{L}_j^{\text{low}}\big\};
	\\
	\Gamma^{\text{med}}&:=\big\{(j,k):\ \min_{i }| D_{j_i}|\geq \delta,\ k\in\mathcal{L}_j^{\text{med}}\big\};
	\\
	\Gamma^{\text{high}}&:=\big\{(j,k):\ \min_{i }| D_{j_i}|\geq \delta,\ k\in\mathcal{L}_j^{\text{high}}\big\}\cup \big\{(j,k):\ \min_{i }| D_{j_i}|< \delta,\ k\in\Z^d\big\}.
\end{aligned}
\end{align}

\begin{lemma}\label{lem:gamma-med}
	Under the conditions of Lemma~\ref{lem:L}, let $0<\delta<1$ and consider the set $\Gamma^{\emph{med}}$ from \eqref{eq_set2}. Then
	\[\#(\Gamma^{\emph{med}})\lesssim \max\{W_{\max},1/\eta_{\partial\ee}\}^{d-1} \cdot \frac{|\partial \ee|}{\kappa_{\partial\ee}} \cdot
	\max\{\log(W_{\max}/\delta),1\}^d
	\cdot s^d .\]
\end{lemma}
\begin{proof}
	By \eqref{eq:lj} and Lemma~\ref{lem:L} it follows
	$$
	\#(\Gamma^{\text{med}})= \sum_{\substack{j\in\Z^d\\ \min|  D_{j_i}|\geq \delta}}\#(\mathcal{L}_j^{\text{med}})\lesssim \sum_{\substack{j\in\Z^d\\ \min|  D_{j_i}|\geq \delta}}\max\{W_{\max},1/\eta_{\partial\ee}\}^{d-1} \frac{|\partial \ee|}{\kappa_{\partial\ee}} s^d.
	$$
	In each coordinate, we have that the number of intervals $  D_{j_i}$ for which $|  D_{j_i}|\geq \delta $ is bounded by
	$C\max\{\log(W_{\max}/\delta),1\}$. Hence, we arrive at
	\[
	\#(\Gamma^{\text{med}})\lesssim \max\{W_{\max},1/\eta_{\partial\ee}\}^{d-1}  \frac{|\partial \ee|}{\kappa_{\partial\ee}} 
	\max\{\log(W_{\max}/\delta),1\}^d
	s^d.\qedhere\]
\end{proof}

\begin{lemma} \label{lem:est-gamma}
	Let $d\geq 2$, $L\ge \max\{W_{\max},1\},$ and $\ee\subseteq \R^d$ be a compact domain with regular boundary at scale $\eta_{\partial\ee}\geq 1$ with constant $\kappa_{\partial\ee}$ and  such that $|\partial E|\ge 1$. Let $s\ge 1$ and $\delta \in (0,1)$ be parameters and consider the sets from \eqref{eq_set2}. Then
	there exists a constant $c=c_\alpha>0$ such that 
	\begin{align}\label{eq:est-gamma-low}
		\sum_{(j,k)\in\Gamma^{\emph{low}}} L^{-d} &\sum_{m\in\ee^c_L}|\mathcal{F}(\Phi_{j,k})(m)|^2  
	\notag	\\ &\lesssim   \max\{W_{\max},1/\eta_{\partial\ee}\}^{d-1} \cdot\frac{|\partial \ee| }{\kappa_{\partial\ee}}\cdot\exp\big(-cs^{1-\alpha}\big)
		\max\{\log(W_{\max}/\delta),1\}^d,
	\end{align}
	and
	\begin{align}\label{eq:est-gamma-high}
		\sum_{(j,k)\in\Gamma^{\emph{high}}} L^{-d} \sum_{m\in\ee_L}|\mathcal{F}(\Phi_{j,k})(m)|^2 \lesssim&    \frac{\max\{W_{\max}, 1/\eta_{\partial\ee}\}^{d-1}}{\kappa_{\partial\ee}}\cdot\Big(|\partial\ee|^{\frac{d}{d-1}} \cdot\delta \notag
		\\
		& + |\partial \ee| \cdot \exp\big(-cs^{1-\alpha}\big)
		\cdot\max\{\log(W_{\max}/\delta),1\}^d
		 \Big).
	\end{align}
\end{lemma}
\proof 
For $j\in\Z^d $ and $l\in\N_0$ we set 
\begin{equation*}
	\mathcal{L}_{j,l}^{\text{low}}=\big\{k\in\Z^d:\   \text{dist}(k,M_j\ee_L^c)\in [s2^l,s2^{l+1})  \big\},
\end{equation*}
and
\begin{equation*}
	\mathcal{L}_{j,l}^{\text{high}}=\big\{k\in\Z^d:\   \text{dist}(k,M_j \ee_L)\in [s2^l,s2^{l+1})  \big\}.
\end{equation*}
Notice that
\begin{align*}
	\mathcal{L}_{j,l}^{\text{low}}\cup \mathcal{L}_{j,l}^{\text{high}}&\subseteq \big\{k\in\Z^d:\ \text{dist}(k,M_j\ee_L)< s2^{l+1} \text{, and } \text{dist}(k,M_j\ee_L^c)< s2^{l+1}\}
	\\ &\subseteq\big\{k\in\Z^d:\ \text{dist}(k,M_j\partial\ee_L)< s2^{l+1}\},
\end{align*}
where the last step follows as in \eqref{eq:lj}. From Lemma \ref{lem:L} we get
\begin{equation}\label{eq:counting}
	\#(\mathcal{L}_{j,l}^{\text{low}})
	,\#(\mathcal{L}_{j,l}^{\text{high}})\lesssim \max\{W_{\max},1/\eta_{\partial\ee}\}^{d-1} \frac{|\partial \ee| }{\kappa_{\partial\ee}  }  s^d2^{dl}.
\end{equation}
From \eqref{eq:decay} it follows that if $k\in\mathcal{L}_{j,l}^{\text{low}}$
\begin{align}
	L^{-d}\sum_{m\in\ee_L^c}|\mathcal{F}(\Phi_{j,k})(m)|^2&\leq  L^{-d} \sum_{m\in\ee_L^c} A_\alpha^{2d}\det(M_j)\exp\big(-2a_\alpha|M_j(m-\xi_{j,k})|^{1-\alpha}\big)\nonumber
	\\
	&\leq A_\alpha^{2d}  L^{-d}\det(M_j) \sum_{m'\in M_j \ee_L^c } \exp\big(-2a_\alpha|m'-k|^{1-\alpha}\big)\nonumber
	\\
	&\lesssim   \int_{\{|x|\ge s2^l\}} \exp\big(-2a_\alpha|x|^{1-\alpha}\big)dx\nonumber
	\\
	&\lesssim    \exp\big(-c(s2^l)^{1-\alpha}\big),\label{eq:est-L-low}
\end{align}
where $c$ can for example be chosen as $a_\alpha$. 
A similar argument also shows that for $k\in \mathcal{L}_{j,l}^{\text{high}}$,
\begin{equation*}\label{eq:est-L-high}
	L^{-d} \sum_{m\in\ee_L}|\mathcal{F}(\Phi_{j,k})(m)|^2\lesssim  \exp\big(-c(s2^l)^{1-\alpha}\big).
\end{equation*}
As $\mathcal{L}_j^{\text{low}}=\bigcup_{l\in\N_0} \mathcal{L}_{j,l}^{\text{low}}$, it follows from \eqref{eq:counting} and \eqref{eq:est-L-low} that
\begin{multline*}
	\sum_{(j,k)\in\Gamma^{\text{low}}}L^{-d}\sum_{m\in\ee_L^c}|\mathcal{F}(\Phi_{j,k})(m)|^2= \sum_{\substack{j\in\Z^d\\ \min|  D_{j_i}|\geq \delta}}\sum_{l\in\N_0}\sum_{k\in\mathcal{L}_{j,l}^{\text{low}}}L^{-d}\sum_{m\in\ee_L^c}|\mathcal{F}(\Phi_{j,k})(m)|^2
	\\
	\begin{aligned}
		&\lesssim \max\{W_{\max},1/\eta_{\partial\ee}\}^{d-1} \frac{|\partial \ee| }{\kappa_{\partial\ee}  } \sum_{\substack{j\in\Z^d\\ \min|  D_{j_i}|\geq \delta}}\sum_{l\in\N_0}
		(s 2^{l})^d  \exp\big(-c(s2^l)^{1-\alpha}\big)
		\\
		&\lesssim \max\{W_{\max},1/\eta_{\partial\ee}\}^{d-1} \frac{|\partial \ee| }{\kappa_{\partial\ee}  }\sum_{\substack{j\in\Z^d\\ \min|  D_{j_i}|\geq \delta}}\exp\big(-c's^{1-\alpha}\big)
		\\
		&\lesssim  \max\{W_{\max},1/\eta_{\partial\ee}\}^{d-1} \frac{|\partial \ee| }{\kappa_{\partial\ee}  }\exp\big(-c's^{1-\alpha}\big)\max\{\log(W_{\max}/\delta),1\}^d,
	\end{aligned}
\end{multline*}
which completes the proof of \eqref{eq:est-gamma-low}.
Again, we can use an analogous reasoning to show
\begin{multline*}
	\sum_{\substack{j\in\Z^d\\ \min|  D_{j_i}|\geq \delta}} \sum_{k\in\mathcal{L}_{j}^{\text{high}}}L^{-d}\sum_{m\in \ee_L}|\mathcal{F}(\Phi_{j,k})(m)|^2
	\\ 
	\lesssim  \max\{W_{\max},1/\eta_{\partial\ee}\}^{d-1} \frac{|\partial \ee| }{\kappa_{\partial\ee}  }\exp\big(-c's^{1-\alpha}\big)\max\{\log(W_{\max}/\delta),1\}^d.
\end{multline*}
Now suppose that $j\in\Z^d$ is such that $\min_{1\le i \le d}|  D_{j_i}|<\delta$.
For every $m\in\Z^d$ we can uniformly bound the following series
$$
\sum_{k\in\Z^d}\exp\big(-2a_\alpha |M_j(m-\xi_{j,k})|^{1-\alpha}\big)=\sum_{k\in\Z^d}\exp\big(-2a_\alpha |M_jm-k|^{1-\alpha}\big)\leq C. 
$$
Since $\det(M_j)=|D_j|$, where $D_j=D_{j_1}\times...\times D_{j_d}$, we thus get by \eqref{eq:decay}
\begin{align*}
	\sum_{\substack{j\in\Z^d\\ \min|  D_{j_i}|< \delta}}\sum_{k\in\Z^d} L^{-d}\sum_{m\in\ee_L} \left|\mathcal{F}(\Phi_{j,k})(m)\right|^2
	&\leq C \sum_{\substack{j\in\Z^d\\ \min|  D_{j_i}|< \delta}} L^{-d}\sum_{m\in\ee_L} \text{det}(M_j)
	\\
	&\leq C L^{-d}\#\ee_L\sum_{\substack{j\in\Z^d\\ \min|  D_{j_i}|< \delta}} |D_j|
	\\
	&\lesssim  \frac{|\partial\ee|^{d/(d-1)}}{\kappa_{\partial\ee}} \sum_{\substack{j\in\Z^d\\ \min|  D_{j_i}|< \delta}} |D_j|,
\end{align*}	
where in the last inequality we used Corollary \ref{cor:coar}.
Finally,
\begin{align*}
	\sum_{\substack{j\in\Z^d\\ \min|  D_{j_i}|< \delta}} |D_j|&\le \sum_{i=1}^d \sum_{\substack{j\in\Z^d\\ |  D_{j_i}|< \delta}} |D_j|
	\le  \sum_{i=1}^d
	 W_{\max}^{d-1} 4\delta
	\\ &\lesssim  \max\{W_{\max},1/\eta_{\partial\ee}\}^{d-1}\delta,
\end{align*}
where we used that each interval $D_{j_i}$ is at most at $|D_{j_i}|<\delta$ distance from the boundary of $(-W_i/2,W_i/2)$ and consequently the union of all such intervals lie in the complement of a rectangle of side-lengths $W_i-4\delta$. This concludes the proof of~\eqref{eq:est-gamma-high}.\pbox

\section{General domain vs. rectangle}\label{sec_gc}

In this section, we prove the following variant of Theorem \ref{th2} for $F$ a rectangle.

\begin{theorem}\label{th:cube}
	Let $L\ge 1$ be a discretization resolution, $d\geq 2$, and $\ee\subseteq\R^d$ be a compact domain with regular boundary at scale $\eta_{\partial\ee}\ge 1$ with constant $\kappa_{\partial\ee}$ and such that $|\partial E|\ge 1$. For $0<W_i\le L$, $i=1,...,d$, take $F=\prod_{i=1}^d(-W_i/2,W_i/2)$ and denote $W_{\max}= \max_i W_i$.
	
	Then for every $\alpha\in(0,1/2)$ there exists $A_{\alpha,d}\geq 1$ such that  for $\varepsilon\in(0,1/2)$:
	\begin{multline*}
		\#\big\{n\in\N:\ \lambda_n\in(\varepsilon,1-\varepsilon)\big\}
		\\ \leq A_{\alpha,d} \cdot \max\{W_{\max},1/\eta_{\partial\ee}\}^{d-1} \cdot \frac{|\partial \ee|}{\kappa_{\partial\ee}  } \cdot \log\left(\frac{  \max\{W_{\max},1/\eta_{\partial\ee}\}^{d-1}|\partial \ee|}{\kappa_{\partial\ee}\ \varepsilon} \right)^{2d(1+\alpha)}.
	\end{multline*}
\end{theorem}

\begin{proof}
	We adopt all the notation of Section \ref{sec_i}. Fix parameters $s \geq 1$, $\delta \in (0,1)$ and consider the sets from \eqref{eq_set2}.
		
	Observe that for $f\in L^2(F)$ one has
	$$
	\|Tf\|_2^2=\|\chi_{F}P_{\ee,L} f\|_2^2\leq \|P_{\ee,L} f\|_2^2=L^{-d}\sum_{m\in \ee_L}\big|\widehat{f}(m)\big|^2,
	$$
	and
	$$
	\|f-Tf\|_2^2=\|\chi_{F}f-\chi_{F} P_{\ee,L} f\|_2^2\leq \|(I-P_{\ee,L}) f\|_2^2=L^{-d}\sum_{m\in \ee_L^c}\big|\widehat f(m)\big|^2.
	$$
	By Lemma~\ref{lem:est-gamma} it thus follows 
	\begin{align}
	\label{eqeps1}
		\sum_{(j,k)\in\Gamma^{\text{low}}}&\|(I-T)\Phi_{j,k}\|_2^2+\sum_{(j,k)\in\Gamma^{\text{high}}}\|T\Phi_{j,k}\|_2^2
		\\
		&\leq C     \frac{\max\{W_{\max},1/\eta_{\partial\ee}\}^{d-1}}{\kappa_{\partial\ee}} \Big(|\partial \ee| \exp\big(-cs^{1-\alpha}\big)\max\{\log(W_{\max}/\delta),1\}^d \notag
		\\ 
		&\hspace{6cm}+ |\partial\ee|^{d/(d-1)}\delta \Big), \notag
	\end{align}
	where the  constants depend only on $\alpha$ and $d$, and $C$ can be taken $\geq 2$.
	
	At last, we can now specify the parameters $\delta$ and $s$ in order for the sets $\Gamma^{\text{low}},\Gamma^{\text{med}}$ and $\Gamma^{\text{high}}$ to play the role of $\mathcal{I}_1,\mathcal{I}_2$ and $\mathcal{I}_3$ in Lemma \ref{lem:main}. We take 
	\begin{align*}
		\delta =\frac{\kappa_{\partial\ee}\ \varepsilon^2 }{C\  \max\{W_{\max},1/\eta_{\partial\ee}\}^{d-1} |\partial\ee|^{d/(d-1)}}.
	\end{align*}
	This ensures that 
	\begin{align}
	\label{eqeps2}
	  \frac{C \max\{W_{\max},1/\eta_{\partial\ee}\}^{d-1} |\partial\ee|^{d/(d-1)}}{\kappa_{\partial\ee}}\delta \leq \varepsilon^2.
	  \end{align}
 while \eqref{eqone} shows that $\delta \in (0,1)$.
 
In addition, we shall select $s$ such that
	\begin{align}
	\label{eqeps3}
		C  \max\{W_{\max},1/\eta_{\partial\ee}\}^{d-1} \frac{|\partial \ee| }{\kappa_{\partial\ee} }\exp\big(-cs^{1-\alpha}\big)\max\{\log(W_{\max}/\delta),1\}^d\leq \varepsilon^2.
	\end{align}
	This condition on $s$ is equivalent to
	\begin{align}\label{eqs}
	    s\hspace{-1pt}\geq \hspace{-1pt}\left(\hspace{-1pt}\frac{1}{c}\log\hspace{-2pt}\left(\frac{C \max\{W_{\max},1/\eta_{\partial\ee}\}^{d-1} |\partial \ee| \max\{\log(W_{\max}/\delta),1\}^d}{\kappa_{\partial\ee}\    \varepsilon^2}\right)\hspace{-2pt}\right)^{\hspace{-2pt}1/(1-\alpha)}.
	\end{align}
 Denote 
 \[u=\max\{W_{\max},1/\eta_{\partial\ee}\}^{d-1}|\partial\ee|/\kappa_{\partial\ee},\]
 which satisfies $u\ge 1$ by \eqref{eqone}, and note that
\begin{align}
    \label{eqdel}
    \log(W_{\max}/\delta) \le \log \Big(\frac{C \kappa_{\partial\ee}^{1/(d-1)} u^{d/(d-1)} }{ \varepsilon^2} \Big)  \lesssim  \log (u/\varepsilon),
\end{align}
where in the last step we used \eqref{eqkappa}.
So, we can bound the right-hand side of \eqref{eqs} by
\[\Big(\frac{1}{c}\log\Big( \frac{C' u \log(u/\varepsilon)^d}{\varepsilon^2}\Big)\Big)^{1/(1-\alpha)}\lesssim \log(u/\varepsilon)^{1/(1-\alpha)}.\]
In particular, \eqref{eqs} is satisfied if 
	\begin{align}\label{eq_ps}
	s= A_{\alpha,d}  \log\left(\frac{  \max\{W_{\max},1/\eta_{\partial\ee}\}^{d-1}|\partial \ee|}{\kappa_{\partial\ee} \ \varepsilon} \right)^{1/(1-\alpha)},
	\end{align}
	for an adequate constant $A_{\alpha,d}$. Moreover, we can guarantee that $s\ge 1$, since, by \eqref{eqone}, the term inside the logarithm in \eqref{eq_ps} is $\ge 2$. From \eqref{eqeps1}, \eqref{eqeps2}, \eqref{eqeps3}, Lemma \ref{lem:main} and Lemma~\ref{lem:gamma-med},
	\begin{align*}
		\# M_\varepsilon(T)&\leq 2^{1-d}\#(\Gamma^{\text{med}}) \\
		&\lesssim \max\{W_{\max},1/\eta_{\partial\ee}\}^{d-1} \frac{|\partial \ee|}{\kappa_{\partial\ee}  } \max\{\log(W_{\max}/\delta),1\}^d s^d
		\\
		&\lesssim  \max\{W_{\max},1/\eta_{\partial\ee}\}^{d-1} \frac{|\partial \ee|}{\kappa_{\partial\ee}  } \log\left(\frac{  \max\{W_{\max},1/\eta_{\partial\ee}\}^{d-1}|\partial \ee|}{\kappa_{\partial\ee} \ \varepsilon} \right)^{d/(1-\alpha)+d}
		\\
		&\lesssim  \max\{W_{\max},1/\eta_{\partial\ee}\}^{d-1} \frac{|\partial \ee|}{\kappa_{\partial\ee}  } \log\left(\frac{  \max\{W_{\max},1/\eta_{\partial\ee}\}^{d-1}|\partial \ee|}{\kappa_{\partial\ee} \ \varepsilon} \right)^{2d(1+\alpha)},
	\end{align*}
 where in the third step we used \eqref{eqdel} and the definition of $s$.
\end{proof}

\section{Eigenvalue estimates for two  domains}\label{sec_step2}

\subsection{Schatten quasi-norm estimates}

For $0<p\le 1$, and $\varepsilon>0$, define the auxiliary function $g=g_{p,\varepsilon}:[0,1]\to\R$ given by
\[g(t)=\Big(\frac{t-t^2}{\varepsilon-\varepsilon^2}\Big)^p.\]
Note that since $\chi_{(\varepsilon,1-\varepsilon)}\le g$, for a positive operator $0\le S\le 1$,
\begin{align}\label{eq_lp}
\# M_\varepsilon(S)=\tr(\chi_{(\varepsilon,1-\varepsilon)} (S))\le \tr(g(S))=\frac{\|S-S^2\|_p^p}{(\varepsilon-\varepsilon^2)^p},
\end{align}
where $\|\cdot\|_p$, $0<p\le 1$, denotes the Schatten quasi-norm.
The next lemma shows that upper bounds for the left-hand side of the last inequality can be transferred to the right-hand side without much loss.

\begin{lemma}
	\label{lem:schp}
	Suppose that for a positive operator $0\le S\le 1$ there are constants $C,D,a>0$ such that for every $\varepsilon\in (0,1/2)$,
	\[\# M_\varepsilon(S)\le C\big(D+\log(\varepsilon^{-1})\big)^a.\]
	Then, for every $0<p\le 1$ there is a constant $C_a>0$ such that
	\[\|S-S^2\|_p^p\le C_a C\big(D+p^{-1}\big)^a.\]
\end{lemma}
\begin{proof}
	By the symmetry of the function $h(x)=x-x^2$ around $1/2$, for $0\le x \le 1$,
	\begin{align*}
		h(x)^p&= \int_0^{\min\{x,1-x\}} (h^p)'(t) dt = \int_0^{1/2} \chi_{(t,1-t)}(x) ph^{p-1}(t) h'(t) dt
		\\& \le \int_0^{1/2} \chi_{(t,1-t)}(x) p t^{p-1} dt.
	\end{align*}
	By a monotone convergence argument we get
	\begin{align*}
		\|S-S^2\|_p^p &\le  \int_0^{1/2} M_t(S) p t^{p-1} dt \le C \int_0^{1} (D+\log(t^{-1}))^a  p t^{p-1} dt
		\\ &= C \int_0^{\infty} (D+u/p)^a   e^{-u} du 
		\\ &\le C (D+1/p)^a + C \int_1^{\infty} (D+u/p)^a   e^{-u} du
		\\ &\le C (D+1/p)^a + C (D+1/p)^a \int_1^{\infty} u^a   e^{-u} du
		\\ &\le(1+\Gamma(a+1)) C (D+1/p)^a. \qedhere
	\end{align*}
\end{proof}

\subsection{Decomposition of the domain and Hankel operators}\label{sec_dec}

In what follows, we let $F\subseteq (-L/2,L/2)^d$ be a compact domain with regular boundary at scale $\eta_{\partial F}=|\partial F|^{1/(d-1)}\ge 1$  with constant $\kappa_{\partial F}$. We also assume that $\mathcal{H}^d(\partial F)=0$ (as otherwise the bounds that we shall derive are trivial). We construct two auxiliary sets $F^-\subseteq F \subseteq F^+$ which will be dyadic approximations of $F$ from the inside and outside by cubes of side-length at least $1$. More precisely, let
\[F=\bigcup_{k\in \Z} \bigcup_{j\in J_k} Q_{k,j},
\qquad \mbox{(up to a null measure set)}
\]
be a dyadic decomposition of $F$ in pairwise disjoint cubes of the form $Q_{k,j}=Q_{2^{k}}+2^kj$ with $k\in\Z$ and $j\in J_k\subseteq \Z^d$, that are maximal (i.e., they are not contained in a larger dyadic cube included in $F$). We define
\[F^-= \bigcup_{k\ge 0} \bigcup_{j\in J_k} Q_{k,j}.\]
For $F^+$ we add cubes of length $1$ to fully cover $F$ and intersect them with $(-L/2,L/2)^d$. The result is a covering of $F$ that combines the cubes from $F^-$ with rectangles of maximal side-length $\le 1$. More precisely, define
\[V=\{v\in \Z^d: \ (F\smallsetminus F^-)\cap (Q_1+v)\neq \varnothing\},\]
and
\[F^+=F^- \cup  \bigcup_{v\in V}\Big(( Q_1+v)\cap (-L/2,L/2)^d\Big)=: F^- \cup  \bigcup_{v\in V} R_v.\]
We write $T^\pm$ for $T_{\ee,F^\pm,L}$. 

We shall invoke Theorem \ref{th:cube} and apply it to each rectangle in the decomposition of $F^-$ and $F^+$. This is allowed because translating the rectangles, or removing their null $\mathcal{H}^d$-measure boundaries, does not affect the eigenvalue estimates in question.

For a set $\kk\subseteq (-L/2,L/2)^d$ define the \emph{Hankel operator} on $L^2((-L/2,L/2)^d)$ by \[H_\kk=(I-P_{\ee,L})\chi_{\kk}P_{\ee,L}\]
and write $H^\pm=H_{F^\pm}$. The Hankel operator $H_\kk$ is related to the \emph{Toeplitz operator} $P_{\ee,L}\chi_{\kk}P_{\ee,L}$ in a similar way as the correspondingly named operators on Hardy spaces \cite{MR1864396}. In particular, we note that
\[(H_\kk)^*H_\kk= P_{\ee,L}\chi_{\kk}P_{\ee,L} -P_{\ee,L}\chi_{\kk}P_{\ee,L}\chi_{\kk}P_{\ee,L}= P_{\ee,L}\chi_{\kk}P_{\ee,L}-(P_{\ee,L}\chi_{\kk}P_{\ee,L})^2.\]
Since $P_{\ee,L}\chi_{\kk}P_{\ee,L}$ and $T_\kk$ share the same non-zero eigenvalues, for $p>0$,
\[\|T_\kk-(T_\kk)^2\|_p^p=\|H_\kk\|_{2p}^{2p}.\]
{Recall that for two operators $S_1,S_2$ in the $p$-Schatten class, $0<p\le 1$,   one has 
	\[ \|{S_1}+{ S_2}\|_{p}^{p}\leq  \| { S_1}\|_{p}^{p}+ \| { S_2}\|_{p}^{p}.\]
}

\begin{lemma}
	\label{lem:pm}
	{Let $L\ge 1$, $d\geq 2$, and $\ee,F\subseteq \R^d$ be  compact domains with regular boundaries at scales $\eta_{\partial\ee}\ge 1$, $\eta_{\partial F}={ |\partial F|^{1/(d-1)} }\ge1$,  with constants $\kappa_{\partial\ee}$, $\kappa_{\partial F}$ respectively.} Assume also that $|\partial E|\ge 1$ and $F\subseteq (-L/2,L/2)^d$.
	
	Then for $\varepsilon\in(0,1/2)$, we have
	\[\# M_\varepsilon(T^\pm)\lesssim   \frac{|\partial \ee|}{\kappa_{\partial \ee}  } \cdot \frac{|\partial F|}{\kappa_{\partial F} } \cdot \log\left(\frac{  |\partial \ee||\partial F|}{\kappa_{\partial \ee} \  \varepsilon}\right)^{2d(1+\alpha)+1} .\]
\end{lemma}
\begin{proof}
	If $k\in\Z$ is such that $J_k\neq \emptyset$, then there is a cube of length $2^k$ included in $F$. In particular, the projection of $\partial F$ onto the hyperplane $\{x_1=0\}$ contains a $(d-1)$-dimensional cube of length $2^k$ and therefore $2^{k(d-1)}\leq |\partial F|$. The maximality of the dyadic decomposition of $F$ implies that $Q_{j,k}\subseteq \partial F+ B_{\sqrt{d}2^{k+1}}(0)$ for  $j\in J_k$. From Lemma \ref{lem:coar} and the fact that $\eta_{\partial F}={|\partial F|^{1/(d-1)}}$, we thus derive
	\begin{align}\label{ecard1}
		2^{dk} \#J_k \le |\partial F+ B_{\sqrt{d}2^{k+1}}(0)| \lesssim 2^k \frac{|\partial F|}{\kappa_{\partial F} }\left(1+\frac{2^{k(d-1)}}{|\partial F|}\right) \lesssim 2^k \frac{|\partial F|}{\kappa_{\partial F} }.
	\end{align}
	Similarly,
	\begin{align}\label{ecard2}
		\# V \le |\partial F+ B_{\sqrt{d}}(0)| \lesssim  \frac{|\partial F|}{\kappa_{\partial F} }.
	\end{align}
	For {$0<2p\le 1$}, and $\varepsilon\in(0,1/2)$, by \eqref{eq_lp}, we get
	\begin{align*}
		\# M_\varepsilon(T^+)&\le \frac{\|T^+-(T^+)^2\|_p^p}{(\varepsilon-\varepsilon^2)^p}= \frac{\|H^+\|_{2p}^{2p}}{(\varepsilon-\varepsilon^2)^p} 
		\\ &\le \left( {2}/{ \varepsilon}\right)^{p} \sum_{k\ge 0} \sum_{j\in J_k} \|H_{Q_{k,j}}\|_{2p}^{2p} + \left( {2}/{ \varepsilon}\right)^{p}  \sum_{v\in V} \|H_{R_v}\|_{2p}^{2p}
		\\ &\lesssim \varepsilon^{-p} \sum_{k\ge 0} \sum_{j\in J_k} \|T_{Q_{k,j}}-T_{Q_{k,j}}^2\|_{p}^{p} + \varepsilon^{-p} \sum_{v\in V} \|T_{R_v}-T_{R_v}^2\|_{p}^{p}.
	\end{align*}
	We invoke Theorem \ref{th:cube} to obtain spectral deviation estimates for each operator $T_{Q_{k,j}}$. These imply that we can apply Lemma \ref{lem:schp} to $T_{Q_{k,j}}$ with $$C\lesssim 2^{k(d-1)}\frac{|\partial\ee|}{\kappa_{\partial\ee}},\quad\text{ and }\quad D=\log\left(2^{k(d-1)}\frac{|\partial\ee|}{\kappa_{\partial\ee}}\right).$$   Similarly, the same holds for $T_{R_v}$ with $k=1$. Choosing $p=\log(2)\big(2\log (\varepsilon^{-1})\big)^{-1}$  {(which ensures that $2p\le 1$ for every $\varepsilon\in(0,1/2)$)} thus yields
	\begin{align*}
		\# M_\varepsilon(T^+)&\hspace{-1pt}\lesssim \hspace{-1pt} \frac{|\partial \ee|}{\kappa_{\partial \ee}  }\hspace{-3pt} \left( \sum_{k\ge 0} \sum_{j\in J_k} 2^{k(d-1)}\hspace{-1pt}  \log\hspace{-2pt}\left(\hspace{-1pt}\frac{  |\partial \ee|2^{k(d-1)}}{\kappa_{\partial \ee}  \ \varepsilon}\hspace{-1pt}\right)^{2d(1+\alpha)} \hspace{-5pt} + \hspace{-1pt} \# V\hspace{-3pt} \cdot\hspace{-2pt} \log\hspace{-2pt}\left(\hspace{-1pt}\frac{  |\partial \ee|}{\kappa_{\partial \ee}  \ \varepsilon}\hspace{-1pt}\right)^{2d(1+\alpha)}\hspace{-1pt}  \right)
		\\ &\hspace{-1pt}\lesssim \hspace{-1pt}  \frac{|\partial \ee|}{\kappa_{\partial \ee}  }  \frac{|\partial F|}{\kappa_{\partial F} }
		\sum_{\substack{k\ge 0 \\ 2^{k(d-1)}\le |\partial F|}} \log\left(\frac{  |\partial \ee|2^{k(d-1)}}{\kappa_{\partial \ee}  \ \varepsilon}\right)^{2d(1+\alpha)}
		\\ &\hspace{-1pt}= \hspace{-1pt}  \frac{|\partial \ee|}{\kappa_{\partial \ee}  }  \frac{|\partial F|}{\kappa_{\partial F} }
		\sum_{0\le k\leq \left\lfloor \frac{\log(|\partial F|)}{\log(2)(d-1)}\right\rfloor}\left( \log\left(\frac{  |\partial \ee| }{\kappa_{\partial \ee}  \ \varepsilon}\right)+(d-1)\log(2)k\right)^{2d(1+\alpha)},
	\end{align*}
	where in the second-to-last step we used \eqref{ecard1}, \eqref{ecard2}, and the fact that $2^{k(d-1)}\leq |\partial F|$ whenever $J_k\neq \emptyset$. Finally, noting that for $C,D,a>0$,
	\[\sum_{k=0}^N (C+Dk)^a\le\int_0^{N+1} (C+Dx)^a dx\le\frac{( C+D(N+1))^{a+1}}{{D}(a+1)},\]
	we get,
	\begin{align*}
		\\  \# M_\varepsilon(T^+)\lesssim \frac{|\partial \ee|}{\kappa_{\partial \ee}  }  \frac{|\partial F|}{\kappa_{\partial F} }  \log\left(\frac{  |\partial \ee| |\partial F|}{\kappa_{\partial \ee}  \ \varepsilon}\right)^{2d(1+\alpha)+1}.
	\end{align*}
	The same argument applies to $ \# M_\varepsilon(T^-)$, and yields the desired conclusion.
\end{proof}

\subsection{The transition index}
The following estimate is part of the proof of \cite[Theorem~1.5]{AbPeRo2} (see also \cite[Lemma~4.3]{maro21}) and allows us to find the index where eigenvalues cross the $1/2$ threshold. We include a proof for the sake of completeness.
\begin{lemma} \label{lem:half}
	For any trace class operator $0\le S\le 1$,
	\begin{enumerate}[label=(\roman*)]
		\item\label{onehalf1} $
		\lambda_n\leq \frac{1}{2},$ for every $n\geq \lceil\tr(S)\rceil+ \max\{2\tr(S-S^2),1\}$; 
		\item\label{onehalf2} $\lambda_n\geq \frac{1}{2},$ for every $1\leq n\leq \lceil \tr(S)\rceil- \max\{2\tr(S-S^2),1\}$.
	\end{enumerate}
\end{lemma}
\begin{proof}
First notice that if $S$ is an orthogonal projection, then the result holds trivially, so we can assume otherwise. In particular, we have that $\tr(S-S^2)>0$.

	Set $K=\lceil\tr(S)\rceil$ and write
	\begin{align*}
		\tr (S)-\tr(S^2)&=\sum_{n=1}^\infty \lambda_n(1-\lambda_n)
		\\
		&=\sum_{n=1}^{K}\lambda_n(1-\lambda_n)+\sum_{n=K+1}^\infty\lambda_n(1-\lambda_n)
		\\
		&\geq \lambda_{K}\sum_{n=1}^{K}(1-\lambda_n)+(1-\lambda_{K})\sum_{n=K+1}^\infty\lambda_n
		\\
		&=\lambda_K K-\lambda_K\sum_{n=1}^{K}\lambda_n+(1-\lambda_{K})\Big(\tr(S)-\sum_{n=1}^{K}\lambda_n\Big)
		\\
		&=\lambda_K K+(1-\lambda_{K})\tr(S)-\sum_{n=1}^{K}\lambda_n
		\\
		&=\tr(S)-\sum_{n=1}^{K}\lambda_n+\lambda_K(K-\tr(S)).
	\end{align*}
	Hence
	\begin{equation}\label{eq:sum1}
		\sum_{n=K+1}^\infty\lambda_n=\tr(S)-\sum_{n=1}^{K}\lambda_n\leq \tr (S)-\tr(S^2),
	\end{equation}
	and 
	\begin{align}\label{eq:sum2}
		\sum_{n=1}^{K-1}(1-\lambda_n)&=\tr(S)-\sum_{n=1}^{K}\lambda_n+\lambda_K(K-\tr(S)) -(1-\lambda_K)(1+\tr(S)-K)\notag
		\\
		&\leq\tr(S)-\sum_{n=1}^{K}\lambda_n+\lambda_K(K-\tr(S))  \leq \tr (S)-\tr(S^2).
	\end{align}
	Now let $j\in\N$ such that $j\ge 2(\tr(S)-\tr(S^2))$ and  consider $k=K+j$ . It follows from \eqref{eq:sum1} that
	$$
	2(\tr(S)-\tr(S^2))\cdot\lambda_k\leq j\cdot\lambda_{K+j}\leq \sum_{n=K+1}^\infty\lambda_n\leq \tr (S)-\tr(S^2),
	$$
	which shows $\lambda_k\leq 1/2$ as $0<\tr(S)-\tr(S^2)<\infty$; this proves part \ref{onehalf1}. 
	
	For part \ref{onehalf2}, if $1\leq k=K-j\leq K-2(\tr(S)-\tr(S^2))$ for $j\in\N$, then \eqref{eq:sum2} implies 
	$$
	2(\tr(S)-\tr(S^2))\cdot(1-\lambda_k)\leq j\cdot (1-\lambda_{K-j})\leq \sum_{n=1}^{K-1}(1-\lambda_n)\leq \tr(S)-\tr(S^2),
	$$
	yielding $\lambda_k\geq 1/2$. This completes the proof.
\end{proof}

\subsection{Proof of the main result}	

With all the preparatory work at hand, we are ready to prove the main result.
	
	\begin{proof}[Proof of Theorem~\ref{th2}]
		Recall from \eqref{eq:stretch} that the eigenvalues of the concentration operator remain the same if we replace $\ee,F$ and $L$ with $t^{-1}\ee,tF$ and $tL$ respectively.  We choose $t=|\partial \ee|^{1/(d-1)}$ and notice that $t^{-1}\ee$ satisfies $|\partial t^{-1}\ee|=1,$ $\eta_{ \partial t^{-1}\ee}=t^{-1}\eta_{\partial \ee}=1$, and $\kappa_{ \partial t^{-1}\ee}=\kappa_{\partial \ee}$. Furthermore, we also have that $tF$ has regular boundary at scale $\eta_{ \partial tF}=  t\eta_{\partial F}=(|\partial\ee||\partial F|)^{1/(d-1)}\geq 1$ with constant $\kappa_{\partial tF}=\kappa_{\partial F}$, and $tL\ge 1$ by assumption on $L$.
		
		Note that for $F'\subseteq (-tL/2,tL/2)^d$, the operator $T$ has integral kernel
		\[K(x,y)=\chi_{F'}(x)\chi_{F'}(y)\frac{1}{(tL)^d}\sum_{k\in(t^{-1}\ee)_{tL}} e^{-2\pi i k (x-y)}.\]
		Thus,
		\[\tr(T)=\int K(x,x)dx=\int_{F'}\frac{1}{(tL)^d}\sum_{k\in (t^{-1}\ee)_{tL}}1 dx=\frac{\#(t   ^{-1}\ee)_{tL}}{(tL)^d}|F'|.\]
		On the other hand, from Lemmas \ref{lem:schp} and \ref{lem:pm} we have that
		\begin{align*}
			\tr\big(T^\pm-(T^\pm)^2\big)&\leq C_{\alpha,d} \frac{|\partial t^{-1}\ee|}{\kappa_{\partial t^{-1} \ee}  } \frac{|\partial tF|}{\kappa_{\partial t F} }  \log\left(\frac{ e |\partial t^{-1} \ee| |\partial tF|}{\kappa_{\partial t^{-1}\ee}   }\right)^{2d(1+\alpha)+1} 
			\\
			&=   C_{\alpha,d} \frac{|\partial \ee|}{\kappa_{\partial \ee}  }  \frac{|\partial F|}{\kappa_{\partial F} }  \log\left(\frac{e  |\partial \ee| |\partial F|}{\kappa_{\partial \ee}   }\right)^{2d(1+\alpha)+1}  =:C_{\ee,F}.
		\end{align*}
	So from Lemma \ref{lem:half},
	\begin{align*}
		\lambda_n(T^+) &\leq \frac{1}{2}, & n &\geq \Big\lceil\frac{\#(t^{-1}\ee)_{tL}}{(tL)^d} |(tF)^+|\Big\rceil+ 2C_{\ee,F}; 
		\\ \lambda_n(T^-) &\geq \frac{1}{2}, &  n&\leq \Big\lceil \frac{\#(t^{-1}\ee)_{tL}}{(tL)^d} |(tF)^-|\Big\rceil- 2C_{\ee,F}.
	\end{align*}
	By Corollary \ref{cor:coar} and $|\partial t^{-1}\ee|=1$,
	\begin{align*}
		\#\{n\in \N: \ \lambda_n(T^-)<1/2,\ \lambda_n(T^+)>1/2\} &\lesssim \frac{1}{\kappa_{\partial\ee}}|(tF)^+\smallsetminus (tF)^-| + C_{\ee,F}
		\\ &\le \frac{1}{\kappa_{\partial\ee}}|\partial tF + B_{\sqrt{d}}(0)| + C_{\ee,F}
		\\ &\lesssim \frac{1}{\kappa_{\partial\ee}} \frac{t^{d-1}|\partial F|}{\kappa_{\partial F}} + C_{\ee,F}
		\lesssim C_{\ee,F},
	\end{align*}
	where in the second to last step we used Lemma \ref{lem:coar}. Since $\lambda_n(T^-)\le \lambda_n(T)\le \lambda_n(T^+)$ for every $n\in \N$, again by Lemma \ref{lem:pm},
	\begin{align*}
		\# M_\varepsilon(T)\leq& \#\{n\in \N: \ 1/2\leq\lambda_n(T^-)<1-\varepsilon\}+\#\{n\in \N: \ \varepsilon<\lambda_n(T^+)\leq 1/2\}
		\\ &+\#\{n\in \N: \ \lambda_n(T^-)<1/2,\ \lambda_n(T^+)>1/2\}		
		\\ \lesssim& \# M_\varepsilon(T^-)+\# M_\varepsilon(T^+) + \#\{n\in \N: \ \lambda_n(T^-)<1/2, \lambda_n(T^+)>1/2\}
		\\\lesssim&   \frac{|\partial \ee|}{\kappa_{\partial \ee}  }  \frac{|\partial F|}{\kappa_{\partial F} } \log\left(\frac{ |\partial \ee|  |\partial F|}{\kappa_{\partial \ee}  \ \varepsilon}\right)^{2d(1+\alpha)+1}. \qedhere
	\end{align*}
\end{proof}

\section{The continuous Fourier transform}\label{sec:con}
In this section we deduce Theorem \ref{th1} from Theorem \ref{th2} by letting $L\to \infty$.

\begin{proof}[Proof of Theorem \ref{th1}]
Fix $\ee$ and $F$ as in the statement of Theorem \ref{th1}. We consider a sufficiently large resolution such that $L\geq |\partial \ee|^{-1/(d-1)}$ and $F\subseteq (-L/2,L/2)^d$.

Let $S_L:L^2(\R^d)\to L^2(\R^d)$ be the operator given by
\[S_L f=T_{\ee,F,L} (\chi_F f)=\chi_F \F_L^{-1}\chi_{\ee_L} \F_L \chi_F f, \qquad f\in L^2(\R^d).\]
Note that $S_L$ and $T_{\ee,F,L}$ share the same non-zero eigenvalues, and recall the operator $S$ from \eqref{eq_S}.

\noindent {\bf Step 1}. We show that
\begin{align}\label{eq_a}
\lim_{L\to \infty} \|S_L-S\|=0.
\end{align}

Recall that $Q_{L^{-1}}=L^{-1}[-1/2,1/2)^d$ and define the auxiliary set
\begin{align*}
	\Gamma_L = 	\bigcup_{m\in\ee_L} m+Q_{L^{-1}}. 
\end{align*}
Note that the symmetric difference $\ee\Delta\Gamma_L$ is included in $\partial\ee+B_{L^{-1}\sqrt{d}}(0)$.
From Lemma \ref{lem:coar},
\begin{align*}
	|\ee\Delta\Gamma_L| \le |\partial\ee+B_{L^{-1}\sqrt{d}}(0)|
	\lesssim \frac{|\partial\ee|}{\kappa L}\big(1+(L\eta_{\partial\ee})^{-(d-1)}\big)\xrightarrow{L\to\infty} 0.
\end{align*}
Using this and setting $R_L=\chi_F \F^{-1}\chi_{\Gamma_L} \F \chi_F$, for $f\in L^2(\R^d)$ we have
\begin{align*}
	\|(R_L-S)f\|_2^2 &\le \|(\chi_{\Gamma_L}-\chi_\ee) \F(\chi_F f) \|_2^2
	\le |\ee\Delta\Gamma_L| \|\F(\chi_F f)\|_\infty^2
	\\ &\le |\ee\Delta\Gamma_L| \|\chi_F f\|_1^2
	\le |\ee\Delta\Gamma_L| |F| \| f\|_2^2\xrightarrow{L\to\infty} 0.
\end{align*}
To prove \eqref{eq_a}, it only remains to show that
\begin{align}\label{eqb}
\|R_L-S_L\|\xrightarrow{L\to\infty} 0.
\end{align}
To this end, let $f\in L^2(\R^d)$ and estimate
\begin{align*}
	\|R_Lf-S_Lf\|_2^2&= \int_F \Big| \int_{\Gamma_L} \F(\chi_F f)(w)e^{2\pi i wx}dw -L^{-d}\sum_{m\in \ee_L}\F(\chi_F f)(m)e^{2\pi i mx}  \Big|^2 dx
	\\&= \int_F \Big|\sum_{m\in \ee_L} \int_{m+Q_{L^{-1}}} \F(\chi_F f)(w)e^{2\pi i wx}-\F(\chi_F f)(m)e^{2\pi i mx} dw \Big|^2 dx
	\\&= \int_F \Big|\sum_{m\in \ee_L} \int_{m+Q_{L^{-1}}} \int_F f(t)\big(e^{2\pi i w(x-t)}-e^{2\pi i m(x-t)}\big)dt dw \Big|^2 dx
	\\&\lesssim \int_F \Big(\sum_{m\in \ee_L} \int_{m+Q_{L^{-1}}} \int_F |f(t)||w-m||x-t| dt dw \Big)^2 dx
	\\&\lesssim \int_F \Big(\sum_{m\in \ee_L} L^{-(d+1)} \int_F |f(t)||x-t| dt \Big)^2 dx
	\\&\lesssim (\#\ee_L)^2 L^{-2(d+1)} \int_F \|f\|_2^2 \int_F |x-t|^2 dt  dx
	\\&\lesssim  L^{-2} \frac{\max\{|\partial \ee|^{2d/(d-1)},1\}}{\kappa_{\partial\ee}^2}|F|^2 \diam(F)^2 \|f\|_2^2,
\end{align*}
where in the last estimate we used Corollary \ref{cor:coar}. Hence \eqref{eqb} holds.

\smallskip

\noindent {\bf Step 2}. Since $S_L$ and $T_{\ee,F,L}$ share the same non-zero eigenvalues, the estimates in Theorem \ref{th2} apply also to $S_L$ for all sufficiently large $L$. By the Fischer-Courant formula, operator convergence of positive compact operators implies convergence of their eigenvalues. Hence, by \eqref{eq_a}, the estimate satisfied by the spectrum of $S_L$ extends to the spectrum of $S$.
\end{proof}

\section{The discrete Fourier transform}\label{sec_dis}

\begin{proof}[Proof of Theorem \ref{th3}]
Let us define $\ee := \Omega+\overline{Q_1}$. Then $\Omega=\ee_L$ for $L=1$.
Let us apply Theorem \ref{th2} with $L=1$ to $\ee$, $F$. We check the relevant hypotheses.

We first note that $\partial E$ is an almost disjoint union of faces of cubes (by almost disjoint we mean that 
 the intersection of any two faces has zero $\H^{d-1}$-measure). Moreover, each one is contained in $k+\overline{Q_1}$ for exactly one $k\in \Omega$ that must belong to $\partial \Omega$. In particular,
\[\partial \ee \subset \bigcup_{k \in \partial \Omega} k + \partial Q_1.\]
Conversely, for each point  $k\in \partial \Omega$ at least one face of the cube $k+\overline{Q_1}$ lies in $\partial \ee$. Thus,
\begin{equation}\label{eq:disc-vs-cont-dom2}
\#\partial\Omega\leq   \big|\partial\ee\big|\leq 2d\cdot\#\partial\Omega,
\end{equation}
and consequently 
\begin{align*}
    |\partial\ee|| \partial F|\geq \#\partial\Omega \cdot | \partial F|\geq 1.
\end{align*}
Moreover, \eqref{eq:disc-vs-cont-dom2} shows that the choice $L=1$ satisfies $L\ge |\partial\ee\big|^{-1/(d-1)}$ as $\partial\Omega$ contains at least one point.

Now fix $0<r \le \sqrt{d}\cdot(\#\partial\Omega)^{1/(d-1)}$ and let us show that $\partial\ee$ is regular at maximal scale. If $r\le 2\sqrt{d}$, and $x\in \partial E$ we clearly have $\H^{d-1}\big(\partial\ee\cap B_r(x)\big)\gtrsim r^{d-1}$ as $\ee$ is a union of cubes of length $1$.  
If $r> 2\sqrt{d}$, set $n=\lfloor r/\sqrt{d} \rfloor$ and let $x\in\partial\ee$.  There exists  $k_x\in \partial \Omega$ such that $|k_x-x|\le \sqrt{d}/2$. Note that for $y\in k_x+\overline{Q_{n+1}}$,
\[|y-x|\le \frac{\sqrt{d} (n+1)}{2}+\frac{\sqrt{d}}{2}\le \frac{r}{2}+\sqrt{d}<r.\]
Hence, $k_x+\overline{Q_{n+1}}\subseteq B_r(x)$. This together with the fact that for $k\in \partial \Omega$ at least one face of the cube $k+\overline{Q_1}$ lies in $\partial \ee$ gives 
\begin{align*}
\H^{d-1}\big(\partial\ee\cap B_r(x)\big)&\ge   \H^{d-1}\big(\partial\ee\cap k_x+\overline{Q_{n+1}}\big)\ge \H^{d-1}\Big(\partial\ee\cap \bigcup_{k\in \partial\Omega  \cap k_x+Q_{n}} k+\partial Q_1\Big)
\\ &\ge \#  \big( \partial\Omega  \cap k_x+Q_{n}\big)
\ge \kappa_{\partial\Omega} n^{d-1} 
\gtrsim \kappa_{\partial\Omega} r^{d-1}.
\end{align*}
This shows that $\partial\ee$  is regular at scale $\sqrt{d}\cdot(\#\partial\Omega)^{1/(d-1)}$ with constant $C_d \cdot\kappa_{\partial\Omega}$. Note that if a set $X$ is regular at scale $\eta_X$ and constant $\kappa_X$, then it is also regular at scale $\alpha\eta_X$ and constant $\min\{1,\alpha^{1-d}\}\kappa_X$, for every $\alpha>0.$
By \eqref{eq:disc-vs-cont-dom2} we therefore see that $\partial\ee$ is regular at scale
$\eta_{\partial\ee}=\big|\partial\ee\big|^{1/(d-1)}$ and constant $\kappa_{\partial\ee}\asymp \kappa_{\partial\Omega}$.

The desired estimates now follow by applying Theorem~\ref{th2} to $\ee$ and $F$ with $L=1$, together with \eqref{eq:disc-vs-cont-dom2}.
\end{proof}

\section{Proof of Remark \ref{rem_N}}\label{sec_rem}
First we combine Lemma~\ref{lem:half}, Lemma~\ref{lem:schp} (for $p=1$) and Theorem~\ref{th1} to conclude that there exist a constant  $C=C_{\alpha,d}>0$ such that if
$$
n\geq \lceil |E|\cdot |F|\rceil + C\frac{|\partial \ee|}{\kappa_{\partial \ee}  }\frac{|\partial F|}{\kappa_{\partial F} } \cdot  \log\left( \frac{e|\partial \ee||\partial F|}{ \kappa_{\partial \ee} }\right)^{2d(1+\alpha)+1}
=:C_1,
$$
then $\lambda_n\leq 1/2$, and if 
$$
n\leq \lceil |E|\cdot |F|\rceil - C\frac{|\partial \ee|}{\kappa_{\partial \ee}  }\frac{|\partial F|}{\kappa_{\partial F} } \cdot \log\left( \frac{e|\partial \ee||\partial F|}{ \kappa_{\partial \ee} }\right)^{2d(1+\alpha)+1}=:C_2,
$$
then $\lambda_n\geq  1/2$. 

For $\varepsilon\in(0,1)$, define $\varepsilon_0:=\min\{\varepsilon, 1-\varepsilon\}\le 1/2$ and let $0<\tau<\varepsilon_0$. Observe that
$$ \{1,...,\lfloor C_2\rfloor\}\smallsetminus M_\tau(S)
 \subseteq N_{1-\varepsilon_0}(S) \subseteq N_\varepsilon(S) \subseteq N_{\varepsilon_0}(S)\subseteq \{1,...,\lceil C_1\rceil\}\cup M_\tau(S),
$$  
where we understand $\{1,...,\lfloor C_2\rfloor\}$ to be $\varnothing$ if $C_2<1$.
Consequently,
\[ C_2-1-\#M_\tau(S) \leq \#N_\varepsilon(S)\leq  C_1 +1+\#M_\tau(S).\]
Rearranging the last expression and using Theorem~\ref{th1} for $\tau$ gives
\[\big|N_\varepsilon(S)-|E|\cdot |F|\big|\lesssim \frac{|\partial \ee|}{\kappa_{\partial \ee}  } \cdot \frac{|\partial F|}{\kappa_{\partial F} } \cdot \log\left( \frac{|\partial \ee||\partial F|}{ \kappa_{\partial \ee}\ \tau}\right)^{2d(1+\alpha)+1}. \]
Letting $\tau \nearrow \varepsilon_0$ yields \eqref{eq_N}.

\end{document}